\documentclass[12pt]{amsart}
\usepackage{abstract}
\usepackage{amssymb,amsmath}
\usepackage{graphicx}
\usepackage{amsmath,amssymb,paralist}
\usepackage{upquote}
\usepackage{amsfonts}
\usepackage{wasysym}
\usepackage{mathbbol}
\usepackage{enumerate}
\usepackage{upquote}
\newcommand{\<}{\langle}
\newcommand{\kt}{\textbf{\textsl{k}}}
\newtheorem{ass}{Condition}
\newtheorem{ass1}{Condition}
\newtheorem{theorem}{Theorem}[section]
\newtheorem{lemma}[theorem]{Lemma}
\newtheorem{proposition}[theorem]{Proposition}
\newtheorem{corollary}[theorem]{Corollary}

\theoremstyle{definition}
\newtheorem{definition}[theorem]{Definition}

\numberwithin{equation}{section}

\begin{document}

\Large

\noindent
\textbf{Theorems of Barth-Lefschetz type and Morse Theory
on the space of paths in Homogeneous spaces}



\vspace{.7cm}

\normalsize

\noindent
\textbf{Chaitanya Senapathi} \newline 
\noindent
Tata Institute of Fundamental Research
\newline 
\noindent
senapath@math.tifr.res.in 

\vspace{.7cm}

\begin{abstract}
Homotopy connectedness theorems for complex submanifolds of homogeneous spaces (sometimes referred to as theorems of Barth-Lefshetz type) have been established by a number of authors. Morse Theory on the space of paths lead to an elegant proof of homotopy connectedness theorems for complex submanifolds of Hermitian symmetric spaces. In this work we extend this proof to a larger class of compact complex homogeneous spaces. 
\end{abstract}


\section{Introduction}
\label{intro}

In the 1920’s Lefschetz \cite{Le} stated the following theorem now known as
the Lefschetz theorem on hyperplane sections. Let $H \subset \mathbb{P}^v$
be a connected complex submanifold of complex dimension $n$. Let $H$ be a hyperplane and
$N \cap H$ be a non singular hyperplane section. Then the relative cohomology
groups satisfy:
$$H^{j}(N,N \cap H ,\mathbb{C}) = 0 \;\;\;\;\;\;\;\; j \leqslant n - 1$$

Fifty years later Barth \cite{B} generalized Lefschetz'theorem: Let $M,N \in \mathbb{P}^{v}$
be complex submanifolds of complex dimensions $m,n$ respectively. If $M$
and $N$ meet properly, then $$H_j(N ,N \cap M) = 0 \text{       ,    } j \leq min(n+m - v; 2m - v+1).$$
Generalizations of Barth's results to homotopy groups were first obtained
by Larsen \cite{La}, Barth-Larsen \cite{B-L} and later by Sommese. Sommese \cite{S1} \cite{S2} and Goldstein \cite{G} generalized these results to submanifolds of generalized flag manifolds, i.e manifolds of the form $G^c/P$ where $G^c$ is a semi-simple complex Lie group and $P$ a parabolic subgroup. Fulton and Lazarsfeld \cite{F-L} proved a stronger version of the result for $\mathbb{P}^{v}$. Later Sommese and Van de Ven \cite{S-V} proved the stronger version of the result for general flag manifolds. 

In 1961 T. Frankel \cite{F} proved a “connectedness” theorem for complex
submanifolds of a K\"{a}hler manifold of positive holomorphic sectional curvature: Let $V$ be a complete K\"{a}hler manifold of positive holomorphic sectional ¨
curvature and of complex dimension $v$. Let $M,N \subset V $ be compact complex
submanifolds of dimensions $m$ and $n$, respectively. If
$m + n \geq v$ then $M$ and $N$ must intersect. Later \cite{K-W}and \cite{S-W} expanded on this idea 
to prove the Barth-Lefschetz theorems on a class of generalized flag manifolds, namely Hermitian Symmetric Spaces and hence reproduced the results of \cite{S1}, \cite{S2} and \cite{G}.

In this paper we extend the theorem of \cite{K-W} and \cite{S-W} to a larger class of generalized flag manifolds. In the main theorem of the work we deal with the case when $G^c$ is simple. 

\begin{theorem}
Let $G^c$ be a simple complex Lie group and $P$ be a parabolic subgroup. Let $V$ be the complex homogeneous space $G^c/P$ with dimension $v$. Let $M,N \subset V $ be compact complex submanifolds dimension $m$ and $n$ respectively. Then there exists a number $\ell$ and a number $\lambda_0 = m+n-(v-\ell)-v$ such that
$$\iota_{*}:\pi_j(N , N\cap M) \rightarrow \pi_j(V,M)$$ is an isomorphism for $j \leq \lambda_0$ and a surjection for $j = \lambda_0+1$.  

\begin{enumerate}[(i)]
 
 \item If $G^c =SL_{r+1}(\mathbb{C})$  then $\ell = r$
 \item If $G^c =SO_{2r}(\mathbb{C})$  then $\ell = 2r-3$
 \item If $G^c =SO_{2r+1}(\mathbb{C})$  then $\ell = 2r-2$
 \item If $G^c =Sp_{r}(\mathbb{C})$  then $\ell = r$
 \item If $G^c = E_6$ ,$E_7$ ,$E_8$ then $\ell= 11 ,17 \text{ and  }29$ respectively 

 \end{enumerate}
\label{realmainthm}
\end{theorem} 

\begin{corollary}
Suppose that $V,M ,$ and $N$ satisfy the same hypotheses as Theorem \ref{realmainthm} and $\ell$ remains the same then 

\begin{enumerate}[(a)]

\item If $j \leq 2m-v-(v-\ell)+1 $ then $\pi_j (V,M) = 0$ 

\item If $j \leq min (2m-v-(v-\ell)+1 ,n+m -v -(v-\ell)) \text{  then  } \newline \pi_j(N ,N \cap M)=0$ 

\end{enumerate}
\end{corollary}

In part (iii) of Theorem \ref{realmainthm}, the case where $G^c = SO_{2r+1}(\mathbb{C})$ and $P$ is the parabolic corresponding to the painted Dynkin diagram with all long roots painted, the result can be improved. By a Theorem of \cite{O} the corresponding homogeneous space can be written in the form $SO_{2r+2}(\mathbb{C})/P$ where $P$ is a parabolic subgroup of $SO_{2r+2}(\mathbb{C})$, in this case $\ell$ can be improved to $2r-1$.

Also the case $G^c = Sp_{r}(\mathbb{C})$ and for a special type of parabolic subgroup $P$, $G^c/P$ is biholomorphic to $CP^{2r-1}$ (\cite{O}) so the number $\ell$ can be improved to $2r-1$. The parabolic P is a maximal parabolic containing a copy of $Sp_{r-1}(\mathbb{C})$.

The results obtained from Theorem \ref{realmainthm} also follow from the work of \cite{S1}, \cite{S2} and \cite{G}. In \cite{S1} and \cite{S2} Sommese shows that the number $\ell$ can be replaced by the co-ampleness of the co-tangent bundle of the respective homogeneous space, and in \cite{G} the co-ampleness of the co-tangent bundle are calculated. In the case where $G^c = Sp_{r}(\mathbb{C})$, $F_4$ or $G_2$ and $P$ is a parabolic such that the corresponding painted Dynkin diagram contains a long root, the approach of \cite{S1}, \cite{S2} and \cite{G} leads to stronger results. All other cases treated in this work, the results are identical to those obtained in \cite{S1}, \cite{S2} and \cite{G}. The approach taken in this work can be applied to $G_2$ and $F_4$ but this would require a tedious case by case analysis of the root systems, so this work has been deferred and not included in the present work. 

To prove these connectedness Theorems \cite{S1} and \cite{S2} apply Morse Theory locally on the ambient space but in this work we apply Morse theory on the space of paths following the work of \cite{F},\cite{S-W},\cite{K-W},\cite{N-W},\cite{Chen},\cite{FM},\cite{Fang},\cite{FMR} and \cite{W}. The basic idea of \cite{S-W} and \cite{K-W} is to demonstrate that the index of the critical points in the space of paths joining two submanifolds has the appropriate lower bound for a chosen Morse function. The Morse function that they choose on the space of paths is the energy function with respect to the invariant symmetric K\"{a}hler metric. To compute a lower bound on the index at the critical points, variational vector fields are constructed along these geodesics and used in the second variation formula.  
        
In this work we generalize the idea of \cite{S-W} and \cite{K-W}. Their argument cannot be generalized to non-symmetric homogeneous spaces as Mok\cite{M1} has shown that other homogeneous spaces that are not covers of products of Hermitian Symmetric Spaces don't posses K\"{a}hler metrics with non-negative curvature. But complex $G^c/P$, being a quotient of a compact Lie group, does posses a metric induced by the standard bi-invariant metric of the compact Lie group. This metric, which is commonly referred to as the 'normal metric' \;is what we use in this paper. This metric has non-negative curvature and is K\"{a}hler only in the case of a Hermitian Symmetric Space. Employing this metric allows us to naturally generalize the work of \cite{S-W} and \cite{K-W}.

Using the canonical connection of the compact homogeneous space, we construct a new connection referred here as the complex-hat connection. We use this connection to form variational vector fields along geodesics. This connection is invariant and is compatible with the complex structure. The connection is also amenable towards the root structure, as a result most of the computations follow naturally. We also use certain types of linear combinations of these variational vector fields and show the existence of a quaternionic structure on the linear combinations. To demonstrate a lower bound on the index, we take an average using this quaternionic structure. Fang \cite{Fang} has employed quaternionic structures to average Hessians, in the context of proving Barth-Lefshetz type theorems for quaternionic submanifolds of positive quaternionic K\"{a}hler spaces.

An outline of this paper is as follows. In Section 2 we review Morse theory on the space of paths and its relation to homotopy theory. In Section 3 we review the basic properties of reductive homogeneous spaces. In Section 4 we construct the complex-hat connection, describe its properties, and describe its relation with the second variation formula. In Section 5 we establish a lower bound on the index of geodesics in terms on an invariant of $G^c$. In the final Section we compute this invariant, thereby proving Theorem \ref{realmainthm}.

\section{Morse theory and Homotopy groups}
In this subsection we follow \cite{S-W} and talk about the Morse theory  on the space of paths, and relate the index of geodesics with the vanishing of relative homotopy groups .

 Let $V$ be a complete Riemannian manifold and let $M$ and $N$ be submanifolds with M compact and N a closed subset of V . We let $P(V ,M, N)$ denote the set of $C^{\infty}$ differentiable  
paths $\gamma : [0; 1] \rightarrow V $ such that $\gamma(0)\in M$
and $\gamma(1) \in N $.  To study the topology of the path space $\Omega(V ,M, N)$ via a Morse function we follow \cite{S-W}. In \cite{S-W}, they generalize Milnor\cite{M}'s approach and approximate the path space by finite-dimensional manifolds and employ techniques from finite-dimensional Morse theory.  Accordingly, in this section, we will describe the general setup, state the results we will need and give the appropriate references in \cite{S-W} and \cite{M}.

Let us denote the set of all piecewise smooth paths from $M$ to $N$ in $V$ by $\Omega(V;M,N)$ or simply $\Omega$. The set $\Omega(V;M,N)$ can be topologized using a natural metric. As a result, the energy of a path given by $$E(\gamma) = \int_{0}^{1} |\dot{\gamma}(t))|^2dt$$ defines a continuous map from $\Omega(V;M,N)\rightarrow \mathbb{R}$.

From \cite{S-W} it is shown that $\gamma$ is a critical point of $E$ if : 
\begin{enumerate}
\item $\gamma$ is a smooth geodesic.
\item $\gamma$ is normal to $M$ and $N$ at $\gamma(0)$ and $\gamma(1)$, respectively. 
\end{enumerate}

Let $T_{\gamma}\Omega$ denote the set piecewise smooth vector fields along $\gamma$ such that vectors at the endpoints belong to the respective submanifolds. This set is in a sense the tangent space for the path $\gamma$ in the space of paths. Let $W_1 ,W_2 \in T_{\gamma}\Omega$. If $\gamma$ is a critical point of $E$ then the second variation of $E$ along $\gamma$, is given by:

$$\frac{1}{2}E_{**}(W_1 ,W_2) = \sum_t<W_2(t),\Delta_t \frac{DW_1}{dt}> - \int_0^1<W_2,\frac{D^2W_1}{dt^2} + R(\dot{\gamma},W_1)\dot{\gamma}>$$

\noindent As in \cite{M} and \cite{S-W}, a finite dimensional approximation $B$ of broken geodesics can be constructed for the set $\Omega_c =E^{-1}([0,c]) \subset \Omega $. 
Let $\Omega^*_c =E^{-1}([0,c]) \subset \Omega$. We restate Theorem 1.2 from \cite{S-W}.

\begin{theorem} 
$E|_{B} :\Omega \rightarrow \mathbb{R}$ is a smooth map. For each $a<c$ the set $B_A = (E_{B})^{-1}([0,a])$ is compact and is a deformation retract of the set $\Omega_a$.The critical points of $E|_{B}$ are precisely the same as the critical points of $E$ in $\Omega^*_c$, that is the smooth geodesics from $M$ to $N$ intersecting $M$ and $N$ orthogonally with energy less than $c$. The index of the hessian of $E|_{B}$ at each such critical point $\gamma$ is equal to the index of $E_{**}$ at $\gamma$. \label{thm2}
\end{theorem} 

\noindent
We now state an elementary lemma from \cite{M}sec 22, about functions on finite-dimensional manifolds: Let $X$ be a smooth manifold and $f:X \rightarrow \mathbb{R}$ be a smooth real-valued function with minimum value 0 such that each $X_c = f^{-1}([0,c])$ is compact. 

\begin{lemma}If the set $X_0$ of minimal points has a neighborhood $U$ with a retraction $r:U \rightarrow X_0$ and if every critical point in $X \ X_0$ has index $> \lambda_0$ then 

$$\pi_j(X,X_0) = 0 \text{   for    } 0 \leq j \leq \lambda_0$$

\end{lemma}

Using the lemma and Theorem \ref{thm2} \cite{S-W} prove the following theorem . 

\begin{theorem} Let $V$ be a complete Complex manifold. Let ¨$M,N \subset V $ be
complex submanifolds and suppose that $M$ is compact and $N$ is a closed
subset of $V$. If every nontrivial critical point of $E$ on $\Omega$ has index  $\lambda > \lambda_0 \geq 0$ ,then the relative homotopy groups $\pi_j(\Omega,\Omega_0) = 0$  for $ 0 \leq j  \leq \lambda_0$.

\end{theorem}

As we can identify $\Omega_0$ with $M \cap N$ we  have $\pi_j(\Omega,M \cap N) = 0$  for $ 0 < j \leq \lambda_0$. Using this observation and the long exact sequence of the pair $(\Omega, N \cap M)$ and some homotopy theory, \cite{S-W} prove the following. 

\begin{theorem}
Let $V$ be a complete complex manifold. Let ¨$M,N \subset V $ be
complex submanifolds and suppose that $M$ is compact and $N$ is a closed
subset of $V$. If every nontrivial critical point of $E$ on $\Omega$ has index  $\lambda > \lambda_0 \geq 0$ ,then the homomorphism induced by the inclusion.

$$\iota_{*}:\pi_j(N , N\cap M) \rightarrow \pi_j(V,M)$$ is an isomorphism for $j \leq \lambda_0$ and a surjection for $j = \lambda_0+1$.

\label{homotopythm}
\end{theorem}

\section{Compact homogeneous spaces}

Let $G$ be a compact semisimple lie group. Let $T_pG$ be the tangent space of $G$ at $p$. We denote the lie algebra by $\mathfrak{g}$ with lie bracket $[\cdot,\cdot]$. We identify $\mathfrak{g}$ with $T_eG$ (the tangent space at the identity).
Let $Ad :G \rightarrow Gl(\mathfrak{g})$ denote the adjoint representation of the group $G$ and let 
 $ad :\mathfrak{g} \rightarrow gl(\mathfrak{g})$ denote the adjoint representation of the lie algebra $\mathfrak{g}$.Then it is well known that $ad(X)Y=[X,Y] \;\; \forall \;\;X,Y\; \in \mathfrak{g}$. Since $G$ is compact and semisimple the Killing form is negative definate. Let $<\cdot,\cdot>$ be the left invariant metric on $G$, which when restricted to the identity is negative of the killing form on  $\mathfrak{g}$. Since the Killing form is invariant under the Adjoint action, the metric satisfies the following property $<X,[Y,Z]> = <[X,Y],Z> \;\; \forall \;\;X,Y,Z\;\; \in \mathfrak{g}$.

Let $K$ be a closed subgroup then $V = G/K $ will be a homogeneous space. Let ${\kt}$ denote the lie subalgebra of $\mathfrak{g}$. Let $\mathfrak{m}$ be the orthogonal compliment of ${\kt}$ and so we have $\mathfrak{g}= {\kt} \oplus \mathfrak{m}$. Since the metric is $Ad$ invariant, in particular it is $Ad_k$ invariant $\forall k \in K$, this implies that $Ad_{|K}(\mathfrak{m})\subseteq \mathfrak{m}$. This makes $V = G/K$ a reductive homogeneous space. This decomposition allows us to identify $T_{\bar{e}}G/K$ with $\mathfrak{m}$ (where $\bar{e} \in G/K$ is the image of the identity by the quotient map $\pi :G \rightarrow G/K$) and this gives us a $K$ action on $T_{\bar{e}}G/K$ which we refer to as the adjoint action. 

Now the group $G$ naturally acts on $G/K$, this action is denoted by $L_g$ for $g \in G$. As a consequence $K$ also acts on $T_{\bar{e}}G/K$, this action is generally known as the isotropy representation. It is easy to see that the isotropy representation and the adjoint action are the same. i.e $\pi_*Ad_k(X)= L_{k*} \pi_*X\;\;\; \forall X \in \mathfrak{m}, \forall k \in K$. 

\subsection{The bracket tensor}
We observe that if a covarint tensor $\omega \in \Lambda^{(1,r)} T_{\bar{e}}G/K $ is $Ad_{|K}$ invariant then $\omega$ can be extended to a $G$ invariant r-tensor. Let $X,Y \in \mathfrak{m}$, denote the $\mathfrak{m}$ component of the bracket by $[X,Y]_{\mathfrak{m}}$. As $Ad(g)$ commutes with the lie bracket $\forall g\in G$, we have the following $Ad^{G}(k)[X,Y]_{\mathfrak{m}} = [Ad^{G}(k)X,Ad^{G}(k)Y]_{\mathfrak{m}}$. Using the identification of $T_{\bar{e}}G/K$ with $\mathfrak{m}$ we have that $[\cdot,\cdot]_{\mathfrak{m}}$ is $Ad_{|K}$ invariant. This global tensor will be denoted by $[\cdot,\cdot]_{\mathfrak{m}}$ and will be referred to as the \textit{bracket tensor}. The bracket tensor also satisfies the following. 
\begin{align} 
<X,[Y,Z]_{\mathfrak{m}}>_p = <[X,Y]_{\mathfrak{m}},Z>_p   \label{associativity}
\end{align}
for $X,Y,X \in T_p(G/K)$. This property will be referred to as \textit{associativity}.
\newline As $Ad^{G}(k)[X,Y]_{{\kt}} = [Ad^{G}(k),Ad^{G}(k)]_{{\kt}}$ and the fact that the metric $<\cdot,\cdot>$ is invariant under $K$, as a result $|[X,Y]_{{\kt}}|^2$ is $Ad_{|K}$ invariant. Thus this $(0,2)$ tensor can also be extended to the whole space.

\subsection{The canonical connection} On a reductive homogeneous space there exists an invariant connection called the 'canonical connection' introduced first by Nomizu \cite{N}. On a compact symmetric space, this connection is the Levi Civita connection. We will briefly go through its construction and basic properties. 

We consider $G$ as a principal fiber bundle over the space $G/K$ with structure group $K$. The action of $K$ on $G$ is right multiplication. The group $G$ itself acts on the fiber bundle, this action clearly commutes with the projection map and the action of $K$. Let us define an equivalence relation '$\sim$' on $G\times \mathfrak{m}$ by  $$(g \times X)  \sim  (gk\times ad^{-1}(k)_*X) \;\;\;\;\ \forall k \in K.$$ Then the quotient space $(G\times \mathfrak{m})$/$\sim$ is nothing but the \textit{associated vector bundle} for the adjoint representation of $K$ on $\mathfrak{m}$. Define $\theta:G \times \mathfrak{m} \rightarrow T(G/K)$ by $\theta(g \times X) = L_{g*}X$. Now $\theta$ clearly factors through the quotient space $(G\times \mathfrak{m})$/$\sim$ so we have an isomorphism of vector bundles.  
$\widetilde{\theta}: (G\times \mathfrak{m})/ \sim \rightarrow T(G/K)$.

We now define a $G$-invariant connection on the principle bundle $G$. We set the horizontal space at the identity to be the space $\mathfrak{m}$ and $L_{g^*}\mathfrak{m}$ at $g$. Using the fact that $Ad_{|K}(\mathfrak{m})\subseteq \mathfrak{m}$, we can clearly see it is compatible with the right action of $K$. The connection on the principal bundle induces a linear connection on the \textit{associated vector bundle} $(G\times \mathfrak{m})/ \sim$ hence also on the tangent space of $G/K$. This linear connection is called the \textit{canonical connection} which we denote by $\nabla$. We now state some important properties of this connection. We refer the reader to ch 10 of \cite{K} for the proofs. 

\begin{theorem}
The parallel transport of $X\in T_{\gamma(0)}G/K$ with respect to the canonical connection along a curve $\gamma :[0,1] \rightarrow G/K $ is given by left translation of some element of $G$, which is independent of $X$. \label{cc2}
\end{theorem}

\begin{proof}
Follows from corollary 2.5 of ch 10 of \cite{K}
 \end{proof} 
By using this theorem and the definition of $G$ invariant tensor we arrive at this corollary. 
\begin{corollary}
Any $G$-invariant tensor is parallel with respect to the canonical connection.
\label{cc3}
\end{corollary}
Let $\gamma_X(t)$ be the integral curve to the left-invariant vector field generated by $X$ in $G$.
\begin{theorem}
$\pi(\gamma_X(t))$ is a geodesic with respect to the canonical connection and all geodesics are of this form or a translate of it. \label{cc21}
\end{theorem}

Let $T(\cdot,\cdot)$ be the torsion of the canonical connection. 
\begin{proposition}
$T(X,Y) = -[X,Y]_{\mathfrak{m}}$ where $T$ is the torsion connection. \label{cc4}
\end{proposition} 

\subsection{Levi-Civita connection}

\begin{definition}
We define a connection $\widetilde{\nabla}$ by
$$\widetilde{\nabla}_XY = \nabla_XY + \frac{1}{2}[X,Y]_\mathfrak{m}$$ 
\end{definition}

\begin{theorem}
\begin{enumerate}

\item The connection $\widetilde{\nabla}$ is the Levi-Civita connection for the metric $<\cdot,\cdot>$.
\item The geodesics of this connection are the same as the geodesics of the canonical connection.
\label{geodesics}
\end{enumerate}
\end{theorem}

\begin{proof}
(1) Since the torsion of the canonical connection is $-[X,Y]_{\mathfrak{m}}$, it is easy to show that this connection has zero torsion. The compatibility of the metric follows from compatibility of the metric with the canonical connection(Cor \ref{cc3}) and the associativity of the  bracket with the metric $<\cdot,\cdot>$ eq(\ref{associativity}).

(2) This follows directly from the formula of the Levi-Civita connection.
 \end{proof}

\begin{theorem}
The curvature of the metric is given by $$<R(X,Y)Y,X> = \frac{1}{4}<[X,Y]_\mathfrak{m},[X,Y]_\mathfrak{m}>+<[X,Y]_{\kt},[X,Y]_{\kt}> \label{curvature}$$
\end{theorem}

\begin{proposition}\label{integrability}
Let $J$ be a  $G$-invariant integrable complex structure on $G/K$. Then 
\begin{enumerate}[a)]    
\item $\nabla J = 0$, where $\nabla$ is the canonical connection.
\item $[X,Y]_{\mathfrak{m}}+J[JX,Y]_{\mathfrak{m}}+J[X,JY]_{\mathfrak{m}}-[JX,JY]_{\mathfrak{m}}=0$
\end{enumerate}
\end{proposition}

\subsection{Complex $G^c/P$} \label{G/P}

Let $G^c$ be a complex semi-simple Lie group and let $\mathfrak{g}_{\mathbb{C}}$ be the corresponding Lie algebra. Let $\mathfrak{h}$ be a Cartan subalgebra. Let $\Delta \subset \mathfrak{h}^*$ be the set of roots. Let $V_{\alpha}^{\mathbb{C}} = \{E \in \mathfrak{g}_{\mathbb{C}}| [h,E] = \alpha(h)E \}$ denote the root space corresponding to $\alpha$. We also have the following decomposition $\mathfrak{g}_{\mathbb{C}}= \mathfrak{h} \oplus \bigoplus_{\alpha \in \Delta}V_{\alpha}^{\mathbb{C}}$.

For a semi-simple Lie algebra $\mathfrak{g}_{\mathbb{C}}$, the roots and root spaces satisfy the following properties. If $\alpha$ is a root then so is $-\alpha$. Each $V_{\alpha}^{\mathbb{C}}$ is one-dimensional. The root space satisfies an important property $[V_{\alpha_1}^{\mathbb{C}},V_{\alpha_2}^{\mathbb{C}}] \subset V_{\alpha_1+\alpha_2}^{\mathbb{C}}$, the bracket is zero if $\alpha_1+\alpha_2$ is not a root and $[V_{\alpha_{1}}^{\mathbb{C}},V_{\alpha_{2}}^{\mathbb{C}}] \subset \mathfrak{h}$ if $\alpha_1+\alpha_2 = 0$.

We can choose a base $\Sigma \subset \Delta$ such that any element of $\Delta$ can be uniquely written as an integer linear combination of elements of $\Sigma$, such that all the co-efficients are either positive or negative. A choice of such a set of roots $\Sigma$, are called simple roots. Let $\Delta^+$, $\Delta^-$ be the set of elements of $\Delta$ that can be written as a positive / negative linear combinations respectively. $\Delta^+$ and $\Delta^-$ will be referred to as positive and negative roots.

We have an inner product on the Lie algebra $\mathfrak{g}_{\mathbb{C}}$ namely the Killing form $\kappa(\cdot,\cdot)$. $\kappa$ is associative in the sense that $\kappa([X,Y],Z) = \kappa([X,[Y,Z])$ for all $X,Y,Z \in \mathfrak{g}$. It also satisfies the following properties $\kappa(V_{\alpha_{1}}^{\mathbb{C}},V_{\alpha_{2}}^{\mathbb{C}}) = 0 $ iff $\alpha_1+\alpha_2 \neq 0$ and $\kappa(V_{\alpha_{i}} ,\mathfrak{h} ) = 0$ for $i = 1,2$, $\alpha_i \in \Delta$. This inner product restricted to $\mathfrak{h}$ is non-degenerate giving us an identification of $\mathfrak{h}$ and $\mathfrak{h}^*$. 

For a given root $\alpha$ we will denote its dual by $t_{\alpha}$. We also denote by $\mathfrak{h}_\mathbb{R}$ the $\mathbb{R}$ linear span of $t_\alpha$ for $\alpha \in \Delta^+$. $\kappa(\cdot,\cdot)$ is positive definite on $\mathfrak{h}_\mathbb{R}$. Let $(\cdot,\cdot)$ be the dual of $\kappa$. Define the structure constants $c_{\alpha,\beta}$ by $[E_\alpha,E_{\beta}]= c_{\alpha,\beta}E_{\alpha+\beta}$. We state a Proposition from [25,chp3 sec 5]

\begin{proposition}\label{Chevally}
We can choose $E_{\alpha}\in V_{\alpha}^{\mathbb{C}}$ such that 
\begin{enumerate}[(a)]
\item $c_{\alpha,\beta} = -c_{\beta,\alpha}$
\item $c_{\alpha,\beta} = -c_{-\alpha,-\beta}$
\item $[E_\alpha,E_{-\alpha}] = t_\alpha$ 
\item $c_{\alpha,\beta} = c_{\beta,-\delta} = c_{-\delta,\alpha}$ whenever $\alpha+\beta = \delta$
\end{enumerate}
 \end{proposition}
 

Let $\Sigma_{{\kt}} \subset \Sigma$ and let $\Delta_{{\kt}}$ be the set of all roots which can be written down as sums of roots of $\Sigma_{{\kt}}$. Let $\Delta_{{\kt}}^+ = \Delta_{{\kt}} \cap \Delta^+$ and $\Delta_{{\kt}}^- = \Delta_{{\kt}} \cap \Delta^- $. Let $\mathfrak{p}= \mathfrak{h}\oplus\bigoplus_{\alpha \in \Delta_{{\kt}}^- \cup \Delta^+}V_{\alpha}^{\mathbb{C}}$. Let $P$ be the Lie subgroup corresponding to the subalgebra $\mathfrak{p}$. $P$ is a parabolic subgroup and every parabolic subgroup is of this form, for an appropriate choice of $\mathfrak{h}$, $\Delta$ and $\Sigma$. \cite{W}

The homogeneous space $V = G^c/P$ is a compact complex homogeneous space and can be written as a quotient $G/K$ where $G$ is a compact subgroup of $G^c$ and $K = G^c \cap P$. We now describe its Lie algebra $\mathfrak{g}$.

Let $X_\alpha = E_\alpha-E_{-\alpha}$ and let $Y_\alpha = iE_{\alpha}+ iE_{\alpha}$. Then $\mathfrak{g}$ decomposes as  $$\mathfrak{g}= i\mathfrak{h}\oplus\bigoplus_{\alpha \in \Delta^+} V_{\alpha}$$where $V_{\alpha}= \text{ span}_{\mathbb{R}} \{X_{\alpha},Y_{\alpha}\}$. $V_{\alpha}$ will be referred to as the real root space associated to $\alpha$.

Since $K = G^c \cap P$, the Lie algebra of $K$ is ${\kt} = i\mathfrak{h}\oplus\bigoplus_{\alpha \in \Delta_{{\kt}}^+} V_{\alpha}$, the restriction of $\kappa$ to $\mathfrak{g}$ is negative definite, hence the corresponding group $G$ is compact. 
Denote by $\< \cdot,\cdot{\rangle}$ the left-invariant metric on $G$ such that the restriction to $\mathfrak{g}$ is $-\kappa_{|\mathfrak{g}}$. Since $\kappa$ is associative, that implies that $\< \cdot,\cdot{\rangle}$ is a bi-invariant metric.  


Let $\Delta_{\mathfrak{m}}$ be the complement of $\Delta_{{\kt}}$ in $\Delta$. Let $\mathfrak{m} = \bigoplus_{\alpha \in \Delta_{\mathfrak{m}}^+} V_{\alpha}$. Using the properties of the Killing form its clear that  $\mathfrak{g} = {\kt} \oplus \mathfrak{m}$ is an orthogonal decomposition. This makes $G/K$ a reductive homogeneous space. We also identify $\mathfrak{m}$ with $T_{\bar{e}}G/K$.

Now we define a complex structure on this tangent space $J:T_{\bar{e}}G/K \rightarrow T_{\bar{e}}G/K $ by (note that $\alpha \in \Delta_{\mathfrak{m}}^+$)
$$J(X_{\alpha}) = Y_{\alpha}$$
$$J(Y_{\alpha})=-X_{\alpha}$$

We note that the complex structure we just defined, implies that $JE_{\alpha} = iE_{\alpha}$ and $JE_{-\alpha} = -iE_{-\alpha}$. This can extended to the whole $G/K$ to give us an invariant, integrable and hermitian complex structure.

\subsection{The map I}

In this section we define a linear operator on a subspace of $\mathfrak{g}$, which along with $J$ gives us quaternionic structure on the subspace. This structure plays an important role in the work below. 

We define a bilinear form on $\mathfrak{m}$ by $R_{Y}X = [Y,X]_{\mathfrak{m}}+J[JY,X]_{\mathfrak{m}}$
 
\begin{lemma} 
\label{mel}
\begin{enumerate}[(a)]
\item If $X,Y \in \mathfrak{m}$ then   
$[X^{1,0},Y^{1,0}]_\mathfrak{m}\in \mathfrak{m}^{1,0}$ 
\item
$R_{Y}X \neq 0 $ iff $[X^{1,0},Y^{0,1}]_\mathfrak{m} \notin \mathfrak{m}^{0,1}$
\item If $X,Y\in \mathfrak{m}$ then $[X^{1,0},Y^{1,0}]_{{\kt}}=0$. 
\item If $[Y,X]_{{\kt}}=0$ and $[Y,JX]_{{\kt}}=0$ iff $[X^{1,0},Y^{0,1}]_{{{\kt}}} = 0$
\end{enumerate}
\end{lemma}

\begin{proof}
\begin{enumerate}[(a)]
\item The result follows by decomposing $X,Y$ into their $\{1,0\}$ and $\{0,1\}$ components in the integrability condition (Proposition \ref{integrability})
\item After a brief calculation we find that $R_{Y}X=-(Z+\overline{Z})$ where $Z = [X^{1,0},Y^{0,1}]_\mathfrak{m}-iJ[X^{1,0},Y^{0,1}]_\mathfrak{m}$ But $Z \in \mathfrak{m}^{1,0}$ so $Z+\overline{Z}=0$  if and only if $Z=0$ if and only if $[X^{1,0},Y^{0,1}]_\mathfrak{m}\in \mathfrak{m}^{0,1}$.
\item It suffices to show that if $\alpha, \beta \in \mathfrak{m}$ the $\alpha + \beta \notin \kt$. This follows trivially from the construction of $\mathfrak{m}$ and ${\kt}$.
\item Using (c) we arrive at $$|[X,Y]_{{\kt}}|^2+|[JX,Y]_{{\kt}}|^2
=4\< [X^{1,0},Y^{0,1}]_{{{\kt}}},[X^{0,1},Y^{1,0}]_{{{\kt}}}{\rangle}$$ The result now follows from this calculation.
\end{enumerate}
 \end{proof}

\begin{lemma} 
\label{1mel}
\begin{enumerate}[(a)]
\item If $[X^{1,0},Y^{0,1}]_{\mathfrak{m}} \in \mathfrak{m}^{0,1}$ then $J[Y,X]_{\mathfrak{m}} = [JY,X]_{\mathfrak{m}}$ 
\item If $[X^{1,0},Y^{0,1}]_{\mathfrak{m}} \in \mathfrak{m}^{0,1}$ then
$J[Y,X]_{\mathfrak{m}} = -[Y,JX]_{\mathfrak{m}} +i[Y^{1,0},X^{1,0}]_{\mathfrak{m}}+i[Y^{0,1},X^{0,1}]_{\mathfrak{m}}$ 
\end{enumerate}
\end{lemma}

\begin{proof}
\begin{enumerate}[(a)]
\item The hypothesis implies $R_{Y}X =0$ from Lemma \ref{mel}. The result now follows directly from the definition of $R_{Y}X $
\end{enumerate}

\vspace{-.5cm}

\begin{align*}
\text{(b)} J[Y,X]_{\mathfrak{m}}+[Y,JX]_{\mathfrak{m}}&= J[Y,X^{1,0}+X^{0,1}]_{\mathfrak{m}}+i[Y,X^{1,0}-X^{0,1}]_{\mathfrak{m}}\\ 
&=i[Y^{1,0},X^{1,0}]_{\mathfrak{m}}+i[Y^{0,1},X^{0,1}]_{\mathfrak{m}}
\end{align*}

The last line follows since $[X^{1,0},Y^{0,1}]_{\mathfrak{m}} \in \mathfrak{m}^{0,1}$  \end{proof}

Let $[\cdot,\cdot]_{\alpha,\beta}$ denote the projection of the bracket on to the subspace $V_{\alpha} \oplus V_{\beta}$. We have following lemma 
\begin{lemma} \label{2mel}
Let $X \in V_{\alpha} \oplus V_{\beta}$ and $\delta = \alpha + \beta$. For $\widetilde{X}_{\delta} = aX_{\delta} + bJX_{\delta}$ $$[\widetilde{X}_{\delta},[\widetilde{X}_{\delta},X]_{\alpha,\beta}]_{\alpha,\beta} = -(a^2+b^2)c_{\alpha,\beta}^2X$$
\end{lemma} 

\begin{proof} 
Using the properties of the structure constants (Proposition \ref{Chevally}) we have 
\begin{align*}
[X_{\delta},X_{\alpha}]
=&[E_\delta-E_{-\delta},E_\alpha-E_{-\alpha}] \\
=&c_{\delta,\alpha}E_{\alpha+\delta}+c_{-\delta,-\alpha}E_{-\alpha-\delta}-c_{-\delta,\alpha}E_{\beta}-c_{\delta,-\alpha}E_{\beta}\\
=&c_{\delta,\alpha}X_{\alpha+\delta}-c_{\delta,-\alpha}X_{\beta}
\end{align*}

\noindent So finally we have $[X_{\delta},X_{\alpha}]_{\alpha,\beta} = -c_{\delta,-\alpha}X_{\beta}$ also have $[X_{\delta},X_{\beta}]_{\alpha,\beta} = -c_{\delta,-\beta}X_{\alpha}$.

As a consequence we have
\begin{align*}
[X_{\delta},[X_{\delta},X_{\alpha}]_{\alpha,\beta}]_{\alpha,\beta} &=-c_{\alpha,\beta}^2X_{\alpha} \\
[X_{\delta},[X_{\delta},X_{\beta}]_{\alpha,\beta}]_{\alpha,\beta} &=-c_{\alpha,\beta}^2X_{\beta}
\end{align*}

Using these calculations and Lemma \ref{1mel} we also observe that 
\begin{align*}
[X_{\delta},[X_{\delta},JX_{\alpha}]_{\alpha,\beta}]_{\alpha,\beta}&=-c_{\alpha,\beta}^2JX_{\alpha}\\
[JX_{\delta},[JX_{\delta},X_{\alpha}]_{\alpha,\beta}]_{\alpha,\beta}&=-c_{\alpha,\beta}^2X_{\alpha}\\
[JX_{\delta},[JX_{\delta},JX_{\alpha}]_{\alpha,\beta}]_{\alpha,\beta}&=-c_{\alpha,\beta}^2JX_{\alpha}
\end{align*}

Using these equations and Lemma \ref{1mel} the lemma follows. 
 \end{proof}

We can define an operator $I_{a,b}:V_{\alpha} \oplus V_{\beta} \rightarrow V_{\alpha} \oplus V_{\beta}$ by $$I_{a,b}X = \frac{1}{(a^2+b^2)^{\frac{1}{2}}|c_{\alpha,\beta}|}[\widetilde{X}_{\delta},X]_{\alpha,\beta}$$

We can readily see that $I_{a,b}^2 = -Id$ and $I_{a,b}JX = -I_{a,b}JX$ which follows from Lemma \ref{1mel} and \ref{2mel}. Let $\mathcal{S}$ be a set of unordered pairs $\{ \alpha,\beta \} \subset \Delta^+$ such that $\alpha+\beta = \delta$. We can naturally extend $I_{a,b}$ to be linear operator on the subspace $$S_0 = \bigoplus_{\{ \alpha,\beta \} \in \mathcal{S}} V_{\alpha} \oplus V_{\beta}$$ 
\begin{lemma}\label{themapI}
The map $I_{a,b}:S_0 \rightarrow S_0$ defined above satisfies the following properties 

\begin{enumerate}[(a)]

\item $I_{a,b}^2 = -Id$

\item $I_{a,b}J = -JI_{a,b}$

\item For $X \in S_0$ we have $|I_{a,b}X|=|X|$

\item If $X \in S_0$ then $\< [I_{a,b}X,X],\widetilde{X}_{\delta}{\rangle}\; \leqslant -N_0(a^2+b^2)^{\frac{1}{2}}|X|^2$ where $N_0$ is a constant only dependent on the Lie algebra $\mathfrak{g}$ where $\widetilde{X}_{\delta} = aX_{\delta}+bJX_{\delta}$
\end{enumerate}
\end{lemma} 

\begin{proof}
For the sake of convenience we refer to $I_{a,b}$ as $I$ in the proof. For (a) and (b) this follows from the preceding discussion.

(c) \begin{align}
\< IX,IX{\rangle}&=\sum_{\alpha+\beta=\delta}\frac{1}{(a^2+b^2)|c_{\alpha,\beta}|^2}\< 
[\widetilde{X}_{\delta},X]_{\alpha,\beta},[\widetilde{X}_{\delta},X]_{\alpha,\beta}{\rangle}\\
&=\sum_{\alpha+\beta=\delta}\frac{-1}{(a^2+b^2)|c_{\alpha,\beta}|^2}\< 
X_{\alpha,\beta},[\widetilde{X}_{\delta},[\widetilde{X}_{\delta},X]_{\alpha,\beta}]{\rangle}\\
&=\sum_{\alpha+\beta=\delta}\< X_{\alpha,\beta},X_{\alpha,\beta}{\rangle}
\end{align}

In the first line we use the fact that the root spaces $V_{\alpha}$ are orthogonal. In the second line we use associativity of bracket (\ref{associativity}), the fact that $\delta+\alpha, \delta+\beta \notin \mathcal{S}$ and also Lemma \ref{2mel}.

(d) 
In the computation we use the associativity of the bracket and part (c) of this lemma. 
\begin{align*}
\< [IX,X],\widetilde{X}_{\delta}{\rangle}
&=-\< IX ,[\widetilde{X}_{\delta},X]{\rangle}\\
&=-\sum_{\{\alpha,\beta\} \in \mathcal{S}}\< IX ,[\widetilde{X}_{\delta},X]_{\alpha,\beta}{\rangle}\\
&=-\sum_{\{\alpha,\beta\} \in \mathcal{S}}{(a^2+b^2)^{\frac{1}{2}}|c_{\alpha,\beta}|} 
|IX_{\alpha,\beta}|^2\\
&=-\sum_{\{\alpha,\beta\} \in \mathcal{S}}{(a^2+b^2)^{\frac{1}{2}}|c_{\alpha,\beta}|} 
|X_{\alpha,\beta}|^2\\
&< -{(a^2+b^2)^{\frac{1}{2}}}N_0 
|X|^2\\
\end{align*}
 \end{proof}

\section{The complex hat connection}

In this Section we assume that $G/K$ is equipped with an invariant integrable complex structure $J$ such that the normal metric is hermitian. Let $M$ and $N$ be two complex submanifolds of dimensions $m,n$ respectively. Let  $\gamma :[0,1] \rightarrow G/K $ be a critical point to the energy functional on the space of paths joining $M$ and $N$ and so it is a geodesic perpendicular to the both manifolds at the endpoints. 

If $X(t)$ is any vector field along the geodesic $\gamma$ such that $X(0)$ and $X(1)$ are in the tangent space of $M$ and $N$ respectively then $X(t)$ is called admissible. For an admissible vector field $X(t)$ we recall the second variation formula\cite{S-W} 

$$E_{**}(X,X) =\< \widetilde{\nabla}_{X}X{\rangle}|_0^1+\int_{0}^{1}
\< \widetilde{\nabla}_{\dot{\gamma}}X,\widetilde{\nabla}_{\dot{\gamma}}X{\rangle}
-\< R(\dot{\gamma},X)X,\dot{\gamma}){\rangle}dt$$  
Where $\widetilde{\nabla}$ is the Levi-Civita connection. Observe that if $X(t)$ is admissible then $JX(t)$ is admissible too, as a consequence the following quadratic form can be defined.

\begin{definition}
The following quantity is defined as the complex energy hessian.
\begin{align}E_{**}^{\mathbb{C}}(X,X) = \frac{1}{2}E_{**}(X,X)+\frac{1}{2}E_{**}(JX,JX)
\end{align}
\end{definition}

\begin{definition}
Define the complex-hat connection $\widehat{\nabla}$  by
$$\widehat{\nabla}_{Y}X = \nabla_{Y}X + \frac{1}{2}R_{Y}(X)$$where $R_{Y}(X) =[Y,X]_\mathfrak{m}+J[JY,X]_\mathfrak{m}$.
\end{definition}

\subsection{The second variation formula}

In this section we rewrite the second variation formula. 
\begin{theorem}
Suppose that $X(t)$ is admissible and parallel with respect to the complex hat connection. Then  
$$E_{**}^{\mathbb{C}}(X,X)=-\int_{0}^{1}\frac{1}{2}(|R_{\dot{\gamma}}(X)|^2)+|[X,\dot{\gamma}]_{\kt}|^2+|[JX,\dot{\gamma}]_{\kt}|^2dt$$
\label{thm.ceh1}
\end{theorem}
\begin{proof}
Using the second variation formula above, we have  
\begin{align} \label{thm.ceh1.1}
E_{**}^{\mathbb{C}}(X,X) &=\< \widetilde{\nabla}_{X}X+\widetilde{\nabla}_{JX}JX,\dot{\gamma}{\rangle}|_0^1-\int_{0}^{1}
\< \widetilde{\nabla}_{\dot{\gamma}}X,\widetilde{\nabla}_{\dot{\gamma}}X{\rangle}+\< \widetilde{\nabla}_{\dot{\gamma}}JX,\widetilde{\nabla}_{\dot{\gamma}}JX{\rangle} \nonumber \\
&-\< R(\dot{\gamma},X)X,\dot{\gamma}){\rangle}-\< R(\dot{\gamma},JX)JX,\dot{\gamma}){\rangle}dt
\end{align}
We begin with the boundary term at $t = 0$. 
\begin{align}
\< \widetilde{\nabla}_{X}X+\widetilde{\nabla}_{JX}JX,\dot{\gamma}{\rangle}_{t=0} &= \< \nabla_{X}X+J\nabla_{JX}X,\dot{\gamma}{\rangle}_{t=0} \nonumber \\
&=\< \nabla_{X}X+J(\nabla_{X}JX+[JX,X]-[JX,X]_{\mathfrak{m}}),\dot{\gamma}{\rangle}_{t=0}\nonumber \\ 
&=\< -J[JX,X]_{\mathfrak{m}}),\dot{\gamma}{\rangle}_{t=0} \label{bdy90210}
\end{align}

In the first line we use the formula for the Levi-Civita connection along with fact that $J$ commutes with $\nabla$. In the next line we use the formula for the torsion of the canonical connection (Proposition \ref{cc4}). For the last line we observe that the variations for $X,JX$ at $t=0$ are tangent to $M$ hence $J[X,JX]$ is tangent too. As $\gamma$ is perpendicular to $M$, we have $\< J([JX,X]),\dot{\gamma}(0){\rangle}=0$.

We can make a similar conclusion for $t=1$ and we have the following.
\begin{align}
\< \widetilde{\nabla}_{X}X+\widetilde{\nabla}_{JX}JX,\dot{\gamma}{\rangle}|_0^1
=\< -J[JX,X]_{\mathfrak{m}},\dot{\gamma}{\rangle}{|_0^1} \label{bdy1}
\end{align}  

The covariant derivative with respect to the canonical connection vanishes for all $G$ invariant tensors (Corollary \ref{cc3}). Applying this to the tensors $J$, $[\cdot,\cdot]_\mathfrak{m}$ and the metric $\< \cdot,\cdot{\rangle}$ we have.
\begin{align}
\< -J[JX,&X]_{\mathfrak{m}},\dot{\gamma}{\rangle}{|_0^1} 
=\int_{0}^{1}\frac{d}{dt}\< J[X,JX]_\mathfrak{m},\dot{\gamma}{\rangle}dt \nonumber \\
&=\int_{0}^{1}\< J[\nabla_{\dot{\gamma}}X,JX]_{\mathfrak{m}}
+J[X,J\nabla_{\dot{\gamma}}X]_\mathfrak{m},\dot{\gamma}{\rangle} + \< J[X,JX]_\mathfrak{m},\nabla_{\dot{\gamma}}\dot{\gamma}{\rangle}dt \nonumber \\
&=\int_{0}^{1}\< \nabla_{\dot{\gamma}}X,[J\dot{\gamma},JX]_{\mathfrak{m}}+J[J\dot{\gamma},X]_\mathfrak{m}{\rangle}dt \label{bdy7}
\end{align}  

\noindent 
In the last line we use the associativity of the bracket (eq \ref{associativity}).

\noindent 
Now let us simplify the second term in equation (\ref{thm.ceh1.1}) and write it in terms of the canonical connection and apply the formula for the curvature (Proposition \ref{curvature})

\begin{align}
&\int_0^1|\widetilde{\nabla}_{\dot{\gamma}}X|^2+|\widetilde{\nabla}_{\dot{\gamma}}JX|^2-\< R(\dot{\gamma},X)X,\dot{\gamma}){\rangle}-\< R(\dot{\gamma},JX)JX,\dot{\gamma}){\rangle} dt\\
&=\int_0^1|\nabla_{\dot{\gamma}}X+\frac{1}{2}[\dot{\gamma},X]_\mathfrak{m}|^2+|J\nabla_{\dot{\gamma}}X+\frac{1}{2}[\dot{\gamma},JX]_\mathfrak{m}|^2-\< R(\dot{\gamma},X)X,\dot{\gamma}){\rangle} \nonumber \\ 
&-\< R(\dot{\gamma},JX)JX,\dot{\gamma}){\rangle}dt\\
&=\int_0^12|\nabla_{\dot{\gamma}}X|^2+\< \nabla_{\dot{\gamma}}X,[\dot{\gamma},X]_\mathfrak{m}{\rangle}-\< \nabla_{\dot{\gamma}}X,J[\dot{\gamma},JX]_\mathfrak{m}{\rangle} \nonumber \\
&-|[X,\dot{\gamma}]_{\kt}|^2-|[JX,\dot{\gamma}]_{\kt}|^2dt 
\label{CEH10}
\end{align}

Now combining eq (\ref{bdy1}) (\ref{bdy7}) and (\ref{CEH10}) in equation (\ref{thm.ceh1.1}) we have
\begin{align}
E_{**}^{\mathbb{C}}&(X,X) \nonumber \\
&=\int_{0}^{1}2|\nabla_{\dot{\gamma}}X|^2+\< \nabla_{\dot{\gamma}}X,[\dot{\gamma},X]_\mathfrak{m}-J[\dot{\gamma},JX]_\mathfrak{m}+[J\dot{\gamma},JX]_{\mathfrak{m}}+J[J\dot{\gamma},X]_\mathfrak{m}{\rangle} \nonumber \\
&-\< [X,\dot{\gamma}]_{\kt}[X,\dot{\gamma}]_{\kt}{\rangle}-\< [JX,\dot{\gamma}]_{\kt},[JX,\dot{\gamma}]_{\kt}{\rangle}dt \nonumber \\
&=\int_{0}^{1}2(|\nabla_{\dot{\gamma}}X|^2+\< \nabla_{\dot{\gamma}}X,[\dot{\gamma},X]_\mathfrak{m}+J[J\dot{\gamma},X]_\mathfrak{m}{\rangle})-\< [X,\dot{\gamma}]_{\kt},[X,\dot{\gamma}]_{\kt}{\rangle} \label{inter} \nonumber \\
&-\< [JX,\dot{\gamma}]_{\kt},[JX,\dot{\gamma}]_{\kt}{\rangle}dt \nonumber \\
&=\int_{0}^{1}2(\< \nabla_{\dot{\gamma}}X,\nabla_{\dot{\gamma}}X+R_{\dot{\gamma}}(X){\rangle}-|[X,\dot{\gamma}]_{\kt}|^2-|[JX,\dot{\gamma}]_{\kt}|^2dt
\end{align}

Note that we use the integrability condition (Proposition \ref{integrability}) in eq (\ref{inter}).  
As $\nabla_{\dot{\gamma}}X+ \frac{1}{2}R_{\dot{\gamma}}(X) =0$, the result follows.
 \end{proof}

\subsection{Properties of the complex-hat connection}

\begin{proposition}
\begin{enumerate}[a)]
  \item The complex-hat connection commutes with the complex structure $J$
  \item The parallel transport along a geodesic $\gamma$ with respect to the complex-hat connection preserves orthogonality between the geodesic and the transported vector. 
  \item If $\widehat{\nabla}_{\dot{\gamma}}X=0$ along a geodesic $\gamma$ and $R_{\dot{\gamma}}(X)$ vanishes at $t=0$ then $\nabla_{\dot{\gamma}}X = 0$ along $\gamma$. 

\end{enumerate}
\label{thm.ceh2}
\end{proposition}

\begin{proof}
\begin{enumerate}[(a)]
  \item
Using part (b) of Proposition \ref{integrability} it is easy to show that $R_{Y}(JX)=JR_{Y}(X)$. The result now follows from this observation and part (a) of Proposition \ref{integrability}.

\item

Let $X(t)$ be a vector field along a geodesic curve $\dot{\gamma}$ which is parallel with respect to the complex-hat connection i.e $\nabla_{\dot{\gamma}}X = -\frac{1}{2}R_{\dot{\gamma}}(X)$ then 
\begin{align}  
\frac{d}{dt}\< X(t),\dot{\gamma}(t){\rangle}&=\< \nabla_{\dot{\gamma}}X,\dot{\gamma}{\rangle}+\< X,\nabla_{\dot{\gamma}}\dot{\gamma}{\rangle} \nonumber \\
&=\< -\tfrac{1}{2}([\dot{\gamma},X]_\mathfrak{m},\dot{\gamma}{\rangle}+\< -J[J\dot{\gamma},X]_\mathfrak{m}),\dot{\gamma}{\rangle} + 0 \nonumber \\
&=\tfrac{1}{2}(\< X,[\dot{\gamma},\dot{\gamma}]_\mathfrak{m}{\rangle}-\< X,[J\dot{\gamma},J\dot{\gamma}]_\mathfrak{m}{\rangle}) \nonumber \\
&=0 \nonumber
\end{align}

In this computation we have used the associativity of the bracket (\ref{associativity}).

\item We can use a canonical connection parallel frame. All invariant tensors are parallel with respect to the canonical connection (Corollary \ref{cc3}) and so is  $\dot{\gamma}(t)$ (Theorem \ref{cc21}). Hence with respect to this frame $R_{\dot{\gamma}}(X)$ is a linear transformation with respect to $X$ with constant co-efficients, so parallel transport with respect to the complex-hat connection can be thought of as the solution to the linear ODE $\dot{X}(t)=-\frac{1}{2}R(X)$. From linear ODE Theory it is clear that $\dot{X}(t)=0$ iff $R(X(0))=0$ .
\end{enumerate}
 \end{proof}

\begin{theorem}
The complex energy hessian of a vector field $X(t)$ which is admissible and parallel with respect to the complex-hat connection is negative iff $[X^{1,0}(0),\dot{\gamma}^{0,1}(0)] \notin \mathfrak{m}^{0,1}$. \label{thm.ceh3}
\end{theorem}

\begin{proof}  
The sufficient condition follows from Theorem \ref{thm.ceh1} and Lemma \ref{mel}. For the necessary condition, now suppose that $[X^{1,0}(0),\dot{\gamma}^{0,1}(0)] \in \mathfrak{m}^{0,1}$. That implies that $R_{\dot{\gamma}}(X(0)) = 0$ and $|[X(0),\dot{\gamma}(0)]_{\kt}|^2+|[JX(0),\dot{\gamma}(0)]_{\kt}|^2 = 0$ by Lemma \ref{mel}. $R_{\dot{\gamma}}(X(0)) = 0$ implies that $X(t)$ is parallel with respect to the canonical connection by part \textit{c}) of Proposition \ref{thm.ceh2}. But that implies that the value $|[X(t),\dot{\gamma}]_{\kt}(t)|^2+|[JX(t),\dot{\gamma}]_{\kt}(t)|^2$ remains zero along the geodesic and that implies that the complex energy hessian is zero.   
 \end{proof}

\section{Index calculations}

In this Section we work with the complex homogeneous space $G^c/P$ where $G^c$ is a complex simple Lie group and $P$ a parabolic subgroup. This manifold can also be written as $G/K$ where $G$ is a compact Lie group and $K$ a closed subgroup. Equip $G/K$ with the normal metric. Let $M$ and $N$ be two complex submanifolds of dimensions $m,n$ respectively, with the dimension of $G/K$ equal to $v$. Let $\gamma :[0,1] \rightarrow G/K$ be a critical point to the energy functional on the space of paths joining $M$ and $N$. Thus $\gamma$ is a geodesic perpendicular to both manifolds at the endpoints. In this Section we will focus on giving a lower bound on the index of each geodesic in terms of $m,n,v$ and an invariant $\ell$ of the Lie algebra $\mathfrak{g}$.

Recall that $\Delta$ and $\Delta_{{\kt}}$ denote the roots associated to the root system of $\mathfrak{g}$ and ${\kt}$ respectively. $\Delta_{\mathfrak{m}}$ denote the roots complementary to $\Delta_{\kt}$. $V_{\alpha} \subset \mathfrak{g}$ is the real root space associated to $\alpha \in \Delta^+$. We define an ordering on $\Delta^+$ by $\alpha <  \delta$ if and only if $\delta-\alpha$ is a positive root. This ordering is not partial.  
After applying an isometry we can identify $T_{\gamma(0)}G/K$ with $\mathfrak{m}$, so $\dot{\gamma}(0) \in \mathfrak{m}$. Define $\Gamma = \{ \alpha \in \Delta_{\mathfrak{m}}^+ | \dot{\gamma}(0) \text{ has a non-trivial component in } V_{\alpha} \} $. Let $\delta \in \Gamma$ be a minimal element such that it also satisfies the following, $\alpha < \beta $ and $\beta < \delta$ implies that $\alpha \notin \Gamma$, we will refer to such $\delta$ as superminimal.    

\begin{align}
\mathcal{S}_{\delta} &= \{ \alpha \in \Delta_{\mathfrak{m}}^+ |\alpha <  \delta \text{ and } \delta - \alpha \in \Delta_{\mathfrak{m}}^+  \} \\
\mathcal{T}_{\delta} &= \{ \beta \in \Delta_{\mathfrak{m}}^+ |\beta \geq \delta \} \cup \{ \alpha \in \Delta_{\mathfrak{m}}^+ |\delta - \alpha \in \Delta_{{\kt}} \} 
\end{align} 

To derive an optimal lower bound on the index we impose the following conditions on the sets $\mathcal{S}_{\delta}$ and $\mathcal{T}_{\delta}$. These conditions will be shown to be true if $\mathfrak{g}$ is a simply laced Lie algebra i.e Lie algebras of the type A, D and E in Section \ref{liecalc}. In the non-simply laced case the conditions are satisfied for most $\delta$. For the remaining choices of $\delta$ the arguments given below can be modified.

\begin{ass}
If $\beta_0 , \beta_1 \in \mathcal{T}_{\delta} \backslash \{ \delta \}$ with $\beta_0 \neq \beta_1$ then $\beta_0 - \delta  \neq \beta_1 -\lambda $ for $\lambda \in \Gamma$ 
\end{ass}

\begin{ass}
If $\alpha, \beta  \in \mathcal{S}_{\delta}$ then $\alpha + \beta \in \Gamma$ iff it is equal to $\delta$.
\end{ass} 

Observe that if $\alpha <  \delta $ then $\alpha , \delta -\alpha$ cannot both lie in $\Delta_{{\kt}}^+$ from the construction of $\kt$, it follows that $|\mathcal{S}_{\delta}|$ is even. Let $\ell = \frac{1}{2}|\mathcal{S}_{\delta}| + |\mathcal{T}_{\delta}|$ and let $h = \frac{1}{2}|\mathcal{S}_{\delta}|$. 

The main goal of this Section is to prove the following theorem. 

\begin{theorem} \label{mainthm}
If conditions 1 and 2 are satisfied then the index of the geodesic $\gamma$ is at least $\mathcal{I} = m+n-(v-\ell)-v+1$   
\end{theorem}

To prove this theorem we will construct a vector space of dimension 4I. We will use a quaternionic structure on this space and show that there exists a subspace of dimension I such that the hessian $E_{**}$ is negative definite. 

Let $$S_0 =\bigoplus_{\alpha \in \mathcal{S}_{\delta}}V_{\alpha} , \;\;\;T_0 =\bigoplus_{\beta \in \mathcal{T}_{\delta}}V_{\beta}$$

The spaces $S_0 ,T_0$ are invariant under $J$ and are of  complex dimension $2h$  and $\ell-h$ respectively. Let $U_0 = S_0 \oplus T_0$ and $\widehat{\tau}:T_{\gamma(0)}(G/K) \rightarrow T_{\gamma(1)}(G/K)$ denote the parallel translation with respect to the complex-hat connection. Let $$U = \{ X \in U_0 \cap T_{\gamma(0)}M |\widehat{\tau}(X) \in T_{\gamma(1)}(N) \} $$. 
\begin{proposition}
The minimum complex dimension of $U$ is $\mathcal{I}+h$.
\end{proposition}
\begin{proof}
Since the dimension of $U_0$ is $\ell+h$ the minimum dimension of $U_0 \cap T_{\gamma(0)}M$ is $\ell+h+m-v$. It is clear that the $$dim_{\mathbb{C}}U \geq dim_{\mathbb{C}}(\widehat{\tau}(U_0 \cap T_{\gamma(0)}M))\cap T_{\gamma(0)}M$$
Since both  $T_{\gamma(0)}M \text{, }T_{\gamma(1)}N$ are perpendicular to $\gamma$ and the complex-hat parallel transport preserves orthogonality with $\gamma$ (part (b) of Proposition \ref{thm.ceh2}) we can get an extra dimension in the count. So the dimension of $U$ is $m+\ell+h-v+n-v+1$. 

 \end{proof}

Let $\dot{\gamma}_{\delta}(0)= aX_{\delta}+bJX_{\delta}$ denote the $V_{\delta}$ component of $\dot{\gamma}(0)$. Lemma \ref{themapI} gives us a linear operator $I_{a,b}:S_0 \rightarrow S_0$ satisfying the following properties, $I^2= -Id$ and $IJ = -JI$ (from now on we omit the subscript). 

Let $S_1 = S_0 \cap U $ and let $dim_{\mathbb{C}}S_1 = s_1$. Using the properties of $I$ we observe that the subspace $S = S_1 \cap IS_1$ is a closed under $I$ and $J$. We see that $ s = dim_{\mathbb{C}}S \geq 2(s_1-h)$ whenever $s_1 > h$ else $S$ may be the trivial space. Now let $T$ be the orthogonal complement of $S_1$ in $U$ and $t =dim_{\mathbb{C}}T$. We observe that $s_1+t \geq \mathcal{I}+h$ and so $t+\frac{1}{2}s \geq \mathcal{I}$.

 Let $\widehat{S}$, $\widehat{T}$ and $\widehat{U}$ be the space of vector fields parallel with respect to the complex-hat connection starting from $S$, $T$ and $U$ respectively. We will use Theorem \ref{thm.ceh3} to show that the index of $E_{**|\widehat{T}}$ is at least t. For $\widehat{S}$ Theorem \ref{thm.ceh3} does not yield anything, instead one must take suitable linear combinations of vector fields in $\widehat{S}$. The index of $E_{**}$ restricted to these linear combinations will be shown to be atleast $\frac{s}{2}$.
 
\subsection{The variations $\widehat{T}$}

\begin{proposition} \label{CEH.T1}
If $Z(t) \in \widehat{T}$ then $E_{**}^{\mathbb{C}}(Z,Z) <  0$ 
\end{proposition}

\begin{proof}  

By Theorem \ref{thm.ceh3} it suffices to show that $[Z(0)^{1,0},\dot{\gamma}^{0,1}(0)] \notin \mathfrak{m}^{0,1}$.
Since $Z(0) \in T$ we can write $Z(0) = X+Y$ where $X \in S$ and $Y \in T$.   
As $[X^{1,0},\dot{\gamma}^{0,1}(0)] \in \mathfrak{m}^{0,1}$ it suffices to prove that $[Y^{1,0},\dot{\gamma}^{0,1}(0)] \notin \mathfrak{m}^{0,1}$.

Suppose $Y^{1,0} = \Sigma_{\beta \in \mathcal{T}_{\delta}}c_{\beta}E_{\beta}$ with $c_{\beta_0} \neq 0$ , for some  $\beta_0 \in \mathcal{T}_{\delta}\backslash \{ \delta \}$ then condition 1 implies that the co-efficient of $E_{\beta_0 - \delta}$ is non-zero in $[Y^{1,0}(0),\dot{\gamma}^{0,1}(0)]$. As either $\beta_0 - \delta >0$ or $|\beta_0 - \delta| \in \Delta_{{\kt}}^+$, implies that $E_{\beta_0 - \delta} \notin \mathfrak{m}^{0,1}$.  


If  $c_{\delta}$ is non-zero and $c_{\beta_0} = 0$ , for all $\beta_0 \in \mathcal{T}_{\delta}\backslash \{ \delta \}$ that implies that $[Y^{1,0},\dot{\gamma}^{0,1}(0)]$ has a non-trivial component in $\mathfrak{h}$, hence $ \notin \mathfrak{m}^{0,1}$ 
 \end{proof}

\begin{theorem} \label{CEH.T}
The index of $\gamma$ when restricted to $\widehat{T}$ is $t$.
\end{theorem}

\begin{proof} 

$\widehat{T}$ has a natural complex structure $J$ on it, since $\widehat{\nabla}$ commutes with $J$. To show that the index is $t$, it suffices to show that every $t+1$ dimensional subspace of the 2$t$ dimensional space $\widehat{T}$ has an element $X$ such that $E_{**}(X,X) <  0$. Let $W$ be such a subspace, then $W \cap JW$ is non-empty. Let $X \in W \cap JW$, then from Proposition  \ref{CEH.T1} we know that the complex energy hessian $E_{**}(X,X) + E_{**}(JX,JX) <  0$. Since both $X,JX \in W \cap JW $ the result follows.
 \end{proof}

\subsection{The variations $\widehat{S}$}
If $\alpha \in \mathcal{S}_{\delta}$ then $\alpha < \delta$ and by the choice of $\delta$ that implies that either $\alpha <  \lambda $ or $\alpha$ is not comparable with $\lambda$ for $\lambda \in \Gamma$. If $X(t) \in \widehat{S}$ then $X(0) \in S$ which implies $[X^{1,0}(0),\dot{\gamma}^{0,1}(0)]_{\mathfrak{m}} \in \mathfrak{m}^{0,1}$. Hence $R_{\dot{\gamma}}(X(0))=0$ so Proposition \ref{thm.ceh2} gives us that $X$ is parallel to the canonical connection.

Let $\tau_0^t:T_{\gamma(0)}G/K \rightarrow T_{\gamma(t)}G/K$ denote parallel translation with respect to the canonical connection. We have $I:S \rightarrow S$ where $S \subset T_{\gamma(0)}M$, using parallel translation we can also define $I_t:\tau_0^t(S) \rightarrow \tau_0^t(S)$. Since the invariant tensors $J$ and $[\cdot,\cdot]_{\mathfrak{m}}$, the vector fields $\dot{\gamma}(t)$ and $X(t)$ are all parallel, the properties of $I$ in Lemma \ref{themapI} carry over to $I_t$. Using this, we can now define a natural operation $I$ on the vector fields $\widehat{S}$ as $IX(t) = I_t(X(t))$
We define $$\widehat{S}_k = \{ Z|\nabla_{\dot{\gamma}}Z = -kIZ \text{ and } Z(0) \in S \}$$ On $\widehat{S}_k$ we can define operations $\bar{I}$ and $\bar{J}$. For $Z \in \widehat{S}_k$ we define these operations via initial conditions,$(\bar{I}Z)(0) = I(Z(0)),(\bar{J}Z)(0) = J(Z(0))$. We observe that $\bar{I}^2 = -Id$,$\bar{J}^2 = -Id$ and $\bar{I}\bar{J} = -\bar{J}\bar{I}$.

\begin{proposition} \label{CEH.S0}
For sufficiently small k. 
\begin{align}
E_{**}(Z,Z)+E_{**}(\bar{I}Z,\bar{I}Z)+E_{**}(\bar{J}Z,\bar{J}Z)+E_{**}(\bar{I}\bar{J}Z,\bar{I}\bar{J}Z)< 0 
\end{align}
\end{proposition}

\begin{proof} 
Let $X \in \widehat{S}$ such that $X(0)=Z(0)$. It is easy to verify that 
\begin{align}
Z &= \cos(kt)X-\sin(kt)IX \label{Z1}\\
\bar{I}Z &= \sin(kt)X+\cos(kt)IX \label{Z2}\\
\bar{J}Z &= \cos(kt)JX+\sin(kt)JX \label{Z3}\\
\bar{I}\bar{J}Z &= -\cos(kt)JIX +\sin(kt)JX \label{Z4}
\end{align}
We see that all these vector fields are admissible since $X$ and $IX$ are. We begin by analyzing $E_{**}(Z,Z)$
\begin{align*}
E_{**}(Z,Z)&=\< \nabla_{Z}Z,\dot{\gamma}{\rangle}|_0^1+\int_0^1|\widetilde{\nabla}_{\dot{\gamma}}Z|^2-\< R(\dot{\gamma},Z)Z,\dot{\gamma}{\rangle}dt\\
&=\< \nabla_{Z}Z,\dot{\gamma}{\rangle}|_0^1+\int_0^1|\nabla_{\dot{\gamma}}Z+\frac{1}{2}[\dot{\gamma},Z]_{\mathfrak{m}}|^2-\frac{1}{4}|[\dot{\gamma},Z]_{\mathfrak{m}}|^2-
|[\dot{\gamma},Z]_{{\kt}}|^2dt\\
&=\< \nabla_{Z}Z,\dot{\gamma}{\rangle}|_0^1+ \int_0^1k^2|IZ|^2-k\< IZ,[\dot{\gamma},Z]_{\mathfrak{m}}{\rangle}-
|[\dot{\gamma},Z]_{{\kt}}|^2dt\\
\end{align*}

If $\alpha,\beta \in \mathcal{S}_{\delta}$ then from condition 2 we have that $\alpha + \beta \notin \Gamma \backslash \{ \delta \}$ and since $\delta$ is superminimal, we have that $\eta < \alpha$ implies $\eta \notin \Gamma$. So we have $\alpha -\beta \notin \Gamma$. We can now conclude that $\< [Z,IZ]_{\mathfrak{m}},\dot{\gamma}(0)-\dot{\gamma}_{\delta}(0){\rangle} = 0$. 

So for sufficiently small $k$ we have 
\begin{align} 
E_{**}(Z,Z)&= \< \nabla_{Z}Z,\dot{\gamma}{\rangle}|_0^1 + \int_0^1k^2|IZ|^2-k\< IZ,[\dot{\gamma}_{\delta},Z]_{\mathfrak{m}}{\rangle}-|[\dot{\gamma},Z]_{{\kt}}|^2 \nonumber \\
&<  \< \nabla_{Z}Z,\dot{\gamma}{\rangle}|_0^1 + \int_0^1k^2|Z|^2-kN_0(a^2+b^2)^{\frac{1}{2}}|Z|^2-|[\dot{\gamma},Z]_{{\kt}}|^2 \nonumber \\
&<  \< \nabla_{Z}Z,\dot{\gamma}{\rangle}|_0^1 \label{CEH.S1}
\end{align}

To get the inequality we use part (d) of Lemma \ref{themapI}. We can get a similar calculation for  $E_{**}(\bar{I}Z,\bar{I}Z),E_{**}(\bar{J}Z,\bar{J}Z) \text{ and }E_{**}(\bar{I}\bar{J}Z,\bar{I}\bar{J}Z)$. We now calculate the corresponding boundary terms.
\begin{align}
\< \nabla_{Z}Z,\dot{\gamma}{\rangle}&+\< \nabla_{\bar{J}Z}\bar{J}Z,\dot{\gamma}{\rangle}+\< \nabla_{\bar{I}Z}\bar{I}Z,\dot{\gamma}{\rangle}+\< \nabla_{\bar{I}\bar{J}Z}\bar{I}\bar{J}Z,\dot{\gamma}{\rangle}|_0^1 \nonumber\\
&=\< \nabla_{X}X,\dot{\gamma}{\rangle}+\< \nabla_{JX}JX,\dot{\gamma}{\rangle}+\< \nabla_{IX}IX,\dot{\gamma}{\rangle}+\< \nabla_{IJX}IJX,\dot{\gamma}{\rangle}|_0^1 \nonumber\\
&=\< [-JX,X]_{\mathfrak{m}},\dot{\gamma}{\rangle}|_0^1+\< [-JIX,IX]_{\mathfrak{m}},\dot{\gamma}{\rangle}|_0^1 \nonumber \\
&=0 \label{CEH.S2}
\end{align}
In the third line we have used equation \ref{bdy90210} from section 4. For the last line we observe that $X$, $[\cdot,\cdot]_{\mathfrak{m}}$ and $J$ are parallel to the canonical connection, which implies that $\< [-JX,X]_{\mathfrak{m}},\dot{\gamma}{\rangle}$ and $\< [-JIX,IX]_{\mathfrak{m}},\dot{\gamma}{\rangle}$ are independent of $t$. 

Using the inequality (\ref{CEH.S1}) and eqn (\ref{CEH.S2}) the proposition is proved. 

 \end{proof}

\begin{theorem}\label{CEH.S} 
For sufficiently small k the index of $\gamma$ when restricted to $\widehat{S}_k$ is $\frac{s}{2}$.
\end{theorem}

\begin{proof} 
The real dimension of $\widehat{S}_k$ is $2s$, so its suffices to show that for any $3\frac{s}{2}+1$ real dimensional subspace $W$, there exists $Z \in \widehat{S}_k$ such that $E_{**}(Z,Z) <  0$. By the properties of $\bar{I}$ and $\bar{J}$ we know that there exist $Z \in W$ such that $Z$, $\bar{I}Z$, $\bar{J}Z$, $\bar{I}\bar{J}Z$ belong to $W$. From Proposition \ref{CEH.S0} the Theorem follows .
 \end{proof} 

\textbf{Note} The content in this subsection is presented so that the reader has adequate motivation for the calculations below. The proof of the Theorem \ref{realmainthm} can be written without reference to the work in this section. 
 
\subsection{Reconciling $\widehat{S}_k$ and $\widehat{T}$}
In the Theorems \ref{CEH.T} and \ref{CEH.S} we have shown the existence of two subspaces of dimension $t$ and $\frac{s}{2}$ belonging to $\widehat{T}$ and $\widehat{S}_k$ such that $E_{**}$ is negative definite. But that is not sufficient to show that the index is $\mathcal{I} = t+\frac{s}{2}$. We will tackle this problem by twisting the space $\widehat{T}$.  

Let $\widehat{T}_k = \{ Z|Z = \cos(kt)X-\sin(kt)Y \text{ with } X,Y \in \widehat{T} \}$. From eqns (\ref{Z1}) to (\ref{Z4}) we observe that $\widehat{S}_k = \{ Z|Z = \cos(kt)W-\sin(kt)IW \text{ with } W \in \widehat{S} \}$. We now define $\widehat{U}_k = \widehat{S}_k \bigoplus \widehat{T}_k$. 
So if $Z \in \widehat{U}_k$ then $Z = cos(kt)(X) - sin(kt)(Y)$ for some $X,Y \in \widehat{U}$.  

We can define two operations on $\widehat{U}_k$, $\bar{I}Z = \sin(kt)(X+W)+\cos(kt)(Y+IW)$ and $\bar{J}Z = \cos(kt)J(X+W)+\sin(kt)J(Y+IW)$. Observe that $\widehat{U}_k$ is closed under the two operations $\bar{I}$ and $\bar{J}$. It can be easily verified that $\bar{I}^2 = -Id$ and $\bar{J}^2 = -Id$ and that $\bar{I}\bar{J} = -\bar{J}\bar{I}$. From eqns (\ref{Z1}) to (\ref{Z4}), we see that the definitions of $\bar{I}$ and $\bar{J}$ coincide. 

\noindent 
Recall the complex energy hessian $E_{**}^{\mathbb{C}}(X,X) = E_{**}(X,X)+E_{**}(JX,JX)$. 
\noindent
For $Z \in \widehat{U}_k$ we  define $$Q(Z) =E_{**}(Z,Z)+E_{**}(\bar{I}Z,\bar{I}Z)+E_{**}(\bar{J}Z,\bar{J}Z)+E_{**}(\bar{J}\bar{I}Z,\bar{J}\bar{I}Z)$$

Now we are ready to prove an important proposition. 

\begin{proposition} 
There exists a $k>0$ such that for any $Z \in \widehat{U}_k$, $Q(Z) <  0$ 
\label{pvf.thm}
\end{proposition}

We first state and prove a few lemmas 

\begin{lemma} \label{CEH.T2}
If $X \in T$ then there exists $M >0$ such that $\frac{1}{2}||R_{\dot{\gamma}}(X(t)|^2+|[X(t),\dot{\gamma}]_{\kt}|^2+|[JX(t),\dot{\gamma}]_{\kt}|^2  > M|X(t)|^2$ for any $t\in [0,1]$ where $X(t)$ is a vector field which is parallel with respect to the complex-hat connection with initial value $X$. As a consequence $E_{**}(X,X)+E_{**}(JX,JX) > M\int_0^1 |X(t)|^2 dt$
\end{lemma}
\begin{proof}
For $X \in T$ we know from the proof of the above Theorem \ref{CEH.T} \newline that either $\frac{1}{2}|R_{\dot{\gamma}}(X(0)| \neq 0$ or $|[X(0),\dot{\gamma}]_{\kt}|^2+|[JX(0),\dot{\gamma}]_{\kt}|^2) \neq 0$. If the first case is true then we know from Proposition \ref{thm.ceh2} part (c) that $R_{\dot{\gamma}}(X(t)) \neq 0$ for all $t$. If the first case were not true then the vector field $X$ is parallel with respect to the canonical connection and so $|[X(t),\dot{\gamma}]_{\kt}|^2+|[JX(t),\dot{\gamma}]_{\kt}|^2$ is the same for all t because $|[\cdot,\cdot]_{\kt}|$ is an invariant tensor. But at $t=0$ we know that this quantity is non-zero and hence it is non-zero for all $t$. 
As the interval $[0,1]$ is compact the lemma follows. 
 \end{proof}

\begin{lemma}
If $Z = \cos(kt)X-\sin(kt)Y$ with $X,Y \in \widehat{U}$ then 
\begin{align*}
Q(Z) &= E_{**}^{\mathbb{C}}(X,X)+E_{**}^{\mathbb{C}}(Y,Y)\\ 
&+\int_{0}^1 2k^2|X|^2+2k^2|Y|^2+2k\< [Y,X]_{\mathfrak{m}}-[JY,JX]_{\mathfrak{m}},\dot{\gamma}{\rangle}dt
\end{align*}
 \label{pvfthm1}
\end{lemma}
\begin{proof}
Let us begin by rewriting the second variational form as the sum of two quadratic forms for ease of computation. 
Now 
\begin{align*}
E_{**}(X,&X)\\
&=\< \widetilde{\nabla}_{X}X,\dot{\gamma}{\rangle}|_0^1+\int_{0}^{1} 
\< \widetilde{\nabla}_{\dot{\gamma}}X,\widetilde{\nabla}_{\dot{\gamma}}X{\rangle}-\< R(\dot{\gamma},X)X,\dot{\gamma}){\rangle}
dt \\
&=\< \widetilde{\nabla}_{X}X,\dot{\gamma}{\rangle}|_0^1+\int_{0}^{1} 
|\nabla_{\dot{\gamma}}X+\frac{1}{2}[\dot{\gamma},X]_\mathfrak{m}|^2-\frac{1}{4}|[X,\dot{\gamma}]_\mathfrak{m}|^2-|[X,\dot{\gamma}]_{\kt}|^2
dt \\
&=\< \widetilde{\nabla}_{X}X,\dot{\gamma}{\rangle}|_0^1+\int_{0}^{1} 
|\nabla_{\dot{\gamma}}X|^2+\< \nabla_{\dot{\gamma}}X,[\dot{\gamma},X]_\mathfrak{m}{\rangle}-|[X,\dot{\gamma}]_{\kt}|^2
dt \\
&=H(X,X)+G(X,X)
\end{align*}
Where $H(X,X)=\< \widetilde{\nabla}_{X}X,\dot{\gamma}{\rangle}|_0^1-\int_{0}^{1}[X,\dot{\gamma}]_{\kt}|^2 dt$ \&
\newline 
$G(X,X)=\int_{0}^{1}|\nabla_{\dot{\gamma}}X|^2+\<\nabla_{\dot{\gamma}}X,[\dot{\gamma},X]_\mathfrak{m}{\rangle}dt$
We observe that $H(\cdot,\cdot)$ is linear with respect to differentiable functions 
and a simple calculation gives us 
\begin{align}
H(Z,Z)+H(\bar{I}Z,\bar{I}Z) = H(X,X)+H(Y,Y) \label{H}
\end{align}

So we focus on $G(Z,Z)+G(\bar{I}Z,\bar{I}Z)$. We start by simplifying 
\begin{align} 
|\nabla_{\dot{\gamma}}Z|^2&+|\nabla_{\dot{\gamma}}\bar{I}Z|^2 
=|\nabla_{\dot{\gamma}}X|^2+|\nabla_{\dot{\gamma}}Y|^2 +k^2|X|^2+k^2|Y|^2-2k\< \cos(kt)\nabla_{\dot{\gamma}}X \nonumber \\
&-\sin(kt)\nabla_{\dot{\gamma}}Y, \bar{I}Z{\rangle} +2k\< \sin(kt)\nabla_{\dot{\gamma}}X+\cos(kt)\nabla_{\dot{\gamma}}Y,Z{\rangle} \nonumber \\ \nonumber
&=|\nabla_{\dot{\gamma}}X|^2+|\nabla_{\dot{\gamma}}Y|^2 +k^2|X|^2+k^2|Y|^2-2k\< \cos^2(kt)\nabla_{\dot{\gamma}}X, Y{\rangle}\\ \nonumber &-2k\< \sin^2(kt)\nabla_{\dot{\gamma}}X, Y{\rangle}+2k\< \cos^2(kt)\nabla_{\dot{\gamma}}Y, X{\rangle}+2k\< \sin^2(kt)\nabla_{\dot{\gamma}}Y, X{\rangle} \nonumber \\
&=|\nabla_{\dot{\gamma}}X|^2+|\nabla_{\dot{\gamma}}Y|^2 +k^2|X|^2+k^2|Y|^2-2k\< \nabla_{\dot{\gamma}}X, Y{\rangle}+2k\< \nabla_{\dot{\gamma}}Y, X{\rangle} \label{G1}
\end{align}
and similar computation leads us to  
\begin{align}
\< \nabla_{\dot{\gamma}}Z,[\dot{\gamma},Z]_{\mathfrak{m}}{\rangle}&+\< \nabla_{\dot{\gamma}}\bar{I}Z,[\dot{\gamma},\bar{I}Z]_{\mathfrak{m}{\rangle}} \nonumber \\
&=\< \nabla_{\dot{\gamma}}X,[\dot{\gamma},X]_{\mathfrak{m}}{\rangle}+\< \nabla_{\dot{\gamma}}Y,[\dot{\gamma},Y]_{\mathfrak{m}}{\rangle}+2k\< [Y,X]_{\mathfrak{m}},\dot{\gamma}{\rangle} \label{G2}
\end{align}

From equation (\ref{G1}) and (\ref{G2}) we arrive at 
\begin{align}
G(Z,Z)+G(\bar{I}Z,\bar{I}Z)& =k\int_0^1 k|X|^2+k|Y|^2-2\< \nabla_{\dot{\gamma}}X, Y{\rangle} \nonumber \\
&+2\< \nabla_{\dot{\gamma}}Y, X{\rangle}+2\< [Y,X]_{\mathfrak{m}},\dot{\gamma}{\rangle}dt \nonumber\\ 
&+ G(X,X)+G(Y,Y) \label{G}
\end{align}
Equations (\ref{H}) and (\ref{G}) give us 
\begin{align} \label{HG}
E_{**}(Z,&Z)+E_{**}(\bar{I}Z,\bar{I}Z) \nonumber \\ =&E_{**}(X,X)+E_{**}(Y,Y)+\int_0^1k^2|X|^2+k^2|Y|^2-2k\< \nabla_{\dot{\gamma}}X, Y{\rangle} \nonumber \\&+2k\< \nabla_{\dot{\gamma}}Y, X{\rangle} +2k\< [Y,X]_{\mathfrak{m}},\dot{\gamma}{\rangle}
\end{align}

Keeping in mind that $\bar{J}Z = \cos(kt)JX+\sin(kt)JY$ and \newline $\bar{J}\bar{I}Z = \cos(kt)JY -\sin(kt)JX$ we can use equation (\ref{HG}) to obtain 
\begin{align*}
E_{**}(\bar{J}Z&,\bar{J}Z)+E_{**}(\bar{I}\bar{J}Z,\bar{I}\bar{J}Z)\\ 
&= E_{**}(JX,JX)+E_{**}(JY,JY)+\int_0^1k^2|JX|^2+k^2|JY|^2\\ 
&-2k\< \nabla_{\dot{\gamma}}JX, J(-Y){\rangle}+2k\< \nabla_{\dot{\gamma}}J(-Y), JX{\rangle}+2k\< [J(-Y),JX]_{\mathfrak{m}},\dot{\gamma {\rangle}} dt\\
&=E_{**}(JX,JX)+E_{**}(JY,JY)+\int_0^1k^2|JX|^2+k^2|JY|^2\\&+2k\< \nabla_{\dot{\gamma}}X,Y{\rangle}-2k\< \nabla_{\dot{\gamma}}Y, X{\rangle}-2k\< [JY,JX]_{\mathfrak{m}},\dot{\gamma}{\rangle} dt
\end{align*}
 The result follows from adding the formulas for $E_{**}(Z,Z)+E_{**}(\bar{I}Z,\bar{I}Z)$ and 
$E_{**}(\bar{J}Z,\bar{J}Z)+E_{**}(\bar{J}\bar{I}Z,\bar{J}\bar{I}Z)$
 \end{proof}

\begin{lemma} \label{themapI2}
Let $P(X(t),Y(t)) = \< [Y(t),X(t)]_{\mathfrak{m}}-[JY(t),JX(t)]_{\mathfrak{m}},\dot{\gamma}{\rangle}$ 
\begin{enumerate}[(a)]
\item There exists a $N$ such that $P(X(t),Y(t))<  N|X(t)||Y(t)|$ for $X,Y \in \widehat{U}_k \; \forall \; t\in \; [0,1]$  
\item $P(W(t),IW(t)) <  -2kN_0(a^2+b^2)^{\frac{1}{2}}|W(t)|^2$
\end{enumerate}
\end{lemma} 

\begin{proof}
\begin{enumerate}[(a)]
\item Clear. 

\end{enumerate}

(b) Since all the objects we deal with are parallel with the canonical connection, it suffices to prove it in the case $t=0$. It is easy to see that $[IW(0),W(0)]_{\mathfrak{m}}-[JIW(0),JW(0)]_{\mathfrak{m}} = 4Re[IW^{1,0},W^{1,0}]_{\mathfrak{m}}$. Now $IW,W \in S_0$ implies that $[W^{1,0},IW^{1,0}]_{\mathfrak{m}} \in V_{\delta}$ via condition 2. As a consequence $\< [W(0),IW(0)]_{\mathfrak{m}}-[JW(0),JIW(0)]_{\mathfrak{m}},\dot{\gamma}(0)-\dot{\gamma}_{\delta}(0){\rangle} =0$. The result now follows form (c) of Lemma \ref{themapI}.

 \end{proof}

\begin{proof} \textit{of Proposition \ref{pvf.thm}}: 
If $Z \in \widehat{U}_k$ there exists $X,Y \in \widehat{T}$ and $W \in \widehat{S}$ so that $Z = cos(kt)(X+W) - sin(kt)(Y+IW)$.
Note that $R_{\dot{\gamma}}W(0) = R_{\dot{\gamma}}(IW(0)) = 0$, so from eqn (\ref{inter}) of Theorem \ref{thm.ceh1} we get that $E_{**}^{\mathbb{C}}(X,X) = E_{**}^{\mathbb{C}}(X+W)$ and $E_{**}^{\mathbb{C}}(Y,Y)=E_{**}^{\mathbb{C}}(Y+IW,Y+IW))$. So now we can rewrite Lemma \ref{pvfthm1}.  
\begin{align*}
Q(Z)&=E_{**}^{\mathbb{C}}(X,X)+E_{**}^{\mathbb{C}}(Y,Y)+\int_0^12k^2(|X+W|^2+|Y+IW|^2)\\&+2k(P(X,Y)+P(X,IW)+P(W,Y)+P(W,IW))dt \nonumber  
\end{align*}
Using part (a) and part (b) of Lemma \ref{themapI2} we have 
 \begin{align*}
 Q(Z) &<  E_{**}^{\mathbb{C}}(X,X)+E_{**}^{\mathbb{C}}(Y,Y)+\int_0^14k^2(|X|^2+|Y|^2)\\ \;\; &+2kN(|X||Y|+|X||W|+|W||Y|)+(8k^2-4k(a^2+b^2)^{\frac{1}{2}})|W|^2dt \\
\end{align*}We choose $\epsilon^2N = 2(a^2+b^2)^{\frac{1}{2}}$ and apply a std inequality.
\begin{align*}
\nonumber Q(Z)&< E_{**}^{\mathbb{C}}(X,X)+E_{**}^{\mathbb{C}}(Y,Y)+\int_0^1(4k^2+kN)(|X|^2+|Y|^2)\\ \;\; &+kN(\frac{|X|^2}{\epsilon^2}+\epsilon^2|W|^2
+\frac{|Y|^2}{\epsilon^2}+\epsilon^2|W|^2)+(8k^2-4k(a^2+b^2)^{\frac{1}{2}})|W|^2dt \\ 
 \nonumber &< \int_0^1(4k^2+kN+\frac{kN}{\epsilon^2})(|X|^2+|Y|^2)+k(8k-2(a^2+b^2)^{\frac{1}{2}})|W|^2dt \\
&+E_{**}^{\mathbb{C}}(X,X)+E_{**}^{\mathbb{C}}(Y,Y)
\end{align*} 
Finally applying Lemma \ref{CEH.T2} we conclude that.
\begin{align*}
Q(Z)&<\int_0^1(4k^2+2kN+\frac{kN}{\epsilon^2}-M)(|X|^2+|Y|^2)\\ \;\; &+k(8k-2(a^2+b^2)^{\frac{1}{2}})|W|^2dt \nonumber 
\end{align*}
It is clear that $k$ and can be chosen small enough so that $Q(Z)$ is negative for all $\bar{Z} \in \widehat{U}_k$. 
 \end{proof}

\begin{proof} \textit{ of Theorem \ref{mainthm}}
To prove this we demonstrate the existence of an $\mathcal{I}$ dimensional subspace of the variation vector fields $\widehat{U}_k$ such that $E_{**}$ is negative definite. We recall that the real dimension of $\widehat{U}_k$ is $4\mathcal{I}$. To show that the index is $\mathcal{I}$ all we have to do is show that for every $3\mathcal{I}+1$ real subspace $W$ there exists a vector field of negative index. As $\bar{J}^2 = -Id $ that implies that in every $3\mathcal{I}+1$ dimensional real space we can find a $2\mathcal{I}+2$ dimensional space which is closed under $\bar{J}$. Now $\bar{I}^2 = -Id$ and $\bar{I}\bar{J} =-\bar{J}\bar{I}$ forces the existence of a 4 dimensional space which is closed under both $\bar{J}$ and $\bar{I}$. 
So we can find a vector field $Z$ such that $\bar{J}Z$ $\bar{I}Z$ and $\bar{J}\bar{I}Z$ belong to $W$. By Theorem \ref{pvf.thm} we know that $Q(Z) <  0$ for uniform k so the hessian is negative for one the following vector fields $Z$, $\bar{J}Z$, $\bar{I}Z$ and $\bar{J}\bar{I}Z$  
 \end{proof}

\section{Lie algebra calculations} 
\label{liecalc}

From Theorem $\ref{mainthm}$ have shown that the index of a geodesic is $\mathcal{I} = m+n-(v-\ell)-v+1$ where $n$, $m$ are the dimensions of the submanifolds $M$, $N$ and $\ell = \frac{1}{2}|\mathcal{S}_{\delta}| + |\mathcal{T}_{\delta}|$, if the roots satisfy certain conditions. For the simply laced Lie algebras $A_r$,$D_r$ and $E_r$ we will show that these conditions are satisfied. For the simple Lie algebra $B_r$ and $C_r$ these conditions are not satisfied when $\delta$ is a short root. In this case we explicitly work with the roots and get over this handicap.

Recall that $\Delta$ is the set of roots for a simple Lie algebra $\mathfrak{g}$. For the rest of this section we assume $\mathfrak{g}$ is not the algebra $G_2$. Let $(\cdot,\cdot)$ be the dual of the Killing form on the Lie algebra $\mathfrak{g}$ restricted to $\mathfrak{h}_{\mathbb{R}}^*$. Normalize $(\cdot,\cdot)$ so that the inner product corresponding to the root system $\Delta$ of $\mathfrak{g}$ is normalized so that the length of each long root is $\sqrt{2}$.

In this Section we refer the reader to \cite{H} for all standard results on root systems and Lie algebras.

\begin{lemma} \label{Weyl-1}
Suppose $\delta$ and $\alpha$ are two unequal roots such that $\delta$ is long then $(\alpha,\delta) > 0$ is $(\alpha,\delta) = 1$
\end{lemma} 
\begin{proof}

 Define $W_{\delta} = \left \{ \alpha \in \Delta | \delta -\alpha \text{ is a root} \right \}$.

Since $(\alpha,\delta) > 0$ the value $2\frac{(\alpha,\delta)}{(\alpha,\alpha)}$ is either 1 or 2. Using table 1 from [H,sec 9.4] it is clear that $2\frac{(\alpha,\delta)}{(\alpha,\alpha)}$ is 1 if $\alpha$ is long and 2 if $\alpha$ is short. In either case it is easy to verify that $(\alpha,\delta) = 1$. 

 \end{proof}

\begin{lemma} \label{Weyl}
If $\delta \in \Delta$ is a long root then, for $\eta_1,\eta_2 \in W_\delta$, $\eta_1 + \eta_2$ is a root if and only if it is equal to $\delta$. 
\end{lemma}

\begin{proof} 

Since the Weyl group is transitive on the set of long roots it suffices to prove this in the case where $\delta$ is the highest root vector. 

Let the $\alpha$-chain of roots through $\delta$ be $$ -p\alpha+\delta, \ldots, \delta, \ldots q\alpha + \delta \; \; \; \text{  where } p,q \geq 0 $$ Then it is well known that $p - q = 2\frac{(\alpha,\delta)}{(\alpha,\alpha)}$. Since $\delta$ is the highest root we have $q=0$. Since $\alpha \in W_\delta$ we have that $p >0$ making $(\alpha ,\delta) > 0$. From Lemma \ref{Weyl-1} $(\alpha ,\delta) = 1$. So if $\eta_1 ,\eta_2 \in W_{\delta}$ then $(\eta_1 + \eta_2,\delta) = 2$. So if $\eta_1+\eta_2$ is a root then it has to be equal to $\delta$. 
 \end{proof}
Recall that 
\begin{align*} 
\mathcal{S}_{\delta} &= \{ \alpha \in \Delta_{\mathfrak{m}}^+ |\alpha <  \delta \text{ and } \delta - \alpha \in \Delta_{\mathfrak{m}}^+  \} \\
\mathcal{T}_{\delta} &= \{ \beta \in \Delta_{\mathfrak{m}}^+ |\beta \geq \delta \} \cup \{ \alpha \in \Delta_{\mathfrak{m}}^+ |\delta - \alpha \in \Delta_{{\kt}} \} 
\end{align*}

\begin{proposition} \label{Weyl1} 
If $\delta$ is a long root then condition 1 and condition 2 are satisfied.
\end{proposition}

\begin{proof} 
We first recall condition 1 and 2 

\begin{ass1}
If $\beta_0, \beta_1 \in \mathcal{T}_{\delta} \backslash \{ \delta \}$ with $\beta_0 \neq \beta_1$ then $\beta_0 - \delta  \neq \beta_1 -\lambda $ for $\lambda \in \Gamma$ 
\end{ass1}

\begin{ass1}
If $\alpha, \beta  \in \mathcal{S}_{\delta}$ then $\alpha + \beta \in \Gamma$ iff it is equal to $\delta$.
\end{ass1}

For condition 1, begin by assuming that $\beta_0 - \delta = \beta_1 -\lambda$. So $\lambda = \beta_1 + \delta -\beta_0$ is a root. Observe that $\beta_1$ and $\delta - \beta_0$ belong to $W_{\delta}$. So by the Lemma \ref{Weyl} the only way that $\beta_1 + \delta - \beta_0 $ is a root is if it is equal to $\delta$, that means $\beta_1 = \beta_0$. A contradiction and hence the lemma readily follows. Condition 2 follows trivially. 
 \end{proof}

\begin{proposition} \label{Weyl2} 
If the root system associated to $\mathfrak{g}^{\mathbb{C}}$ is such that all roots are long, the value $\ell = \frac{1}{2}|\mathcal{S}_{\delta}| + |\mathcal{T}_{\delta}|$ is independent of $\delta$ and only dependent on the lie algebra $\mathfrak{g}^{\mathbb{C}}$. The values of $\ell$ are given below 

\begin{enumerate}[(a)]
 \item If $\mathfrak{g}^{\mathbb{C}} =\mathfrak{sl}_{r+1}(\mathbb{C})$  then $\ell = r$
\item If $\mathfrak{g}^{\mathbb{C}} =\mathfrak{so}_{2r}(\mathbb{C})$  then $\ell = 2r-3$
\item If $\mathfrak{g}^{\mathbb{C}} = E_6$ ,$E_7$ ,$E_8$ then $\ell= 11 ,17 \text{ and }29$ respectively 
\end{enumerate}

\end{proposition} 

\begin{proof}

We begin by showing that $\ell$ is independent of $\delta$ and later calculate ' $\ell$' by choosing $\delta$ appropriately.

Define $\widetilde{W}_{\delta} = \left \{  \{\alpha,\beta  \}| \alpha + \beta = \delta, \alpha,\beta \in \Delta \right \}$. It is clear that $|\widetilde{W}_{\delta}| +1 = \frac{1}{2}|\mathcal{S}_{\delta}| + |\mathcal{T}_{\delta}|$. Since the Weyl group acts linearly $|\widetilde{W}_{\delta}|$ is preserved by the Weyl group. As these Lie algebras are simply-laced, it is well known that the Weyl group acts transitively on the roots \cite{H} therefore $|\widetilde{W}_{\delta}|$ is independent of $\delta$. We observe that  $|W_{\delta}| = 2|\widetilde{W}_{\delta}|$ and calculate $|W_{\delta}|$ for a suitable root $\delta$ in each case. We refer to \cite{H} for a description of the respective root systems that we will use.

$\textbf{A}_r$ : \text{The roots are given by } $\Delta = \{ e_i-e_j | 1 \leq i,j \leq r+1 \}$ where $e_i \in \mathbb{R}^{r+1}$ are the standard basis. Let us choose $\delta = e_1 -e_2 $. Then it is clear that $W_{\delta} = \{e_1-e_j| 2 < j \} \cup \{e_k-e_2 | k \leq r+1 \} $. Thus $|W_{\delta}| =  2r-2$,so $\ell = r$

$\textbf{D}_r$ : \text{The roots are given by } $\Delta = \{ \pm e_i \pm e_j |i \neq j, 1 \leq i,j \leq r+1 \}$ where $e_i \in \mathbb{R}^{r+1}$ are the standard basis. Choose $\delta = e_1 -e_2 $. we then have that $W_{\delta} = \{e_1\pm e_j ,e_k \pm e_2 | 2 <  j,k \leq r \} $. So $|W_{\delta}| =  4r-8$

$\textbf{E}_8$ :\text{The roots are given by } \newline $$\Delta = \{ \pm e_i \pm e_j |i \neq j,\text{and } 1 \leq i,j \leq 8 \} \cup \left \{ \frac{1}{2}\sum_{i=1}^8 t_ie_i |t_i= \pm 1 , \prod_{i=1}^8t_i = 1 \right \}$$ where $e_i \in \mathbb{R}^{8}$. Choose $\delta = e_1 -e_2 $. A little computation gives us that 
\begin{align*}
W_{\delta} &= \{e_1\pm e_j| 2 < j\} \cup \{e_k \pm e_2 | k \leq 8 \} \\&\cup \left \{ \frac{1}{2}(e_1-e_2)+{}\frac{1}{2}\sum_{i=3}^8 t_ie_i \Big |t_i= \pm 1 , \prod_{i=3}^8 t_i = 1 \right \}
\end{align*}

\noindent So we have $|W_{\delta}| = 56$ and so $\ell = 29$. A similar calculation for $\textit{E}_6$ \text{ and } $\textit{E}_7$ \text{ can be made}. 

 \end{proof}

\noindent \textit{Proof of Theorem \ref{realmainthm}}

\noindent \textit{Parts i,ii and v} 

In Lie algebras of type A, D and E all roots are long, so parts $i$, $ii$ and $iv$ follow from Theorem \ref{mainthm}, Proposition \ref{Weyl1} and Propsition \ref{Weyl2}.

\noindent \textit{Part iii}) 

\noindent Referring back to \cite{H} we see that the roots of $B_r$ are given by $$\Delta = \{ \pm e_i \pm e_j | 1 \leq i,j \leq r \} \cup \{ \pm e_i | 1 \leq i \leq r\}$$ where $e_i$ are the standard basis of $\mathbb{R}^r$. The positive roots are given by $$\Delta^+ = \{ e_i + e_j | 1 \leq i,j \leq r \} \cup \{ e_i - e_j | 1 \leq i < j \leq r \} \cup \{ e_i | 1 \leq i \leq r\}.$$ Recall that given a $\Gamma \subset \Delta^+$ we chose $\delta \in \Gamma$ to be a superminimal element i.e a minimal element such that  if $\alpha < \beta $ and $\beta < \delta$ then $\alpha \notin \Gamma$. 

Suppose we can choose a superminimal element $\delta$ such that it is a long root. In this case we can repeat the arguments of the the simply laced case, using Proposition \ref{Weyl1} and Theorem \ref{mainthm}. We can make similar calculations and conclude that $\ell = 2r-2$. 

Suppose no superminimal elements are long. In this case we must choose $\delta = e_i$ where $i$ is the largest index such that $e_i \in \Gamma$. Observe that $e_l +e_j \notin \Gamma$ for $i < l < j$ and that $e_l - e_j \notin \Gamma$ for $l < j$. For if any of these statements were not true it would imply that there exists a superminimal long root. 

We now claim that if $e_k \notin \Delta_{\kt} $ for $i < k \leq r$ then $\delta = e_i$ satisfies conditions 1 and 2. To verify the claim we argue as follows. It is clear that $\mathcal{S}_\delta$ is a subset of the following $\{ e_i-e_l| i < l \} \cup \{ e_l|\;i < l \}$. Thus condition 2 can be easily verified. Let $U_\delta =\{ \beta \geq \delta | \beta \in \Delta_{\mathfrak{m}}^+ \}$ and let $V_\delta =\{ \alpha \in \Delta_{\mathfrak{m}}^+ | \delta -\alpha \in \Delta_{\kt}^+ \}$. Then $\mathcal{T}_\delta = U_\delta \cup V_\delta$. We can clearly see that $U_\delta = \{ e_a| a \ leq i\} \cup \{e_i+e_a |\; a<i\}$. Due to the assumption that $e_k \notin \Delta_{\kt}$, we have $V_{\delta}= \{e_b|\; i<b,e_i-e_b \in \Delta^+_{\kt} \} $. Thus $$\mathcal{T}_\delta = U_\delta \cup V_\delta = \{ e_a| a \leq i\} \cup \{e_i+e_a |\; a<i\} \cup \{e_b|\; i<b,e_i-e_b \in \Delta^+_{\kt}\}$$ 
Note that if $\beta_1,\beta_2 \in \mathcal{T}_{\delta}$ then $\beta_1 + \delta - \beta_2$ is a positive root if and only if its equal to $\delta$ or of the form $e_s - e_t$ where $s < t$. To see this note that 
$$\{\delta - \beta| \beta \in \mathcal{T}_\delta\} =\{ -e_a+e_i| a < i \} \cup \{ -e_a | a < i \} \cup \{ e_i-e_b |\;i<b,e_i-e_b \in \Delta^+_{\kt} \}$$ 
We have already observed that roots of the form $e_s - e_t$ do not belong to $\Gamma$ with $s < t$, we now see that condition 1 is satisfied thereby verifying the claim. 
 
To finish the proof we assume that $e_k \in \Delta_{\kt}$ for $k$ such that $i < k$. Let $g_t = exp(tX_{e_k})$ where $X_{e_k}$ is a root vector for the root $e_k$. Since $X_{e_k} \in K$ that implies that $g_t \in K$. Then the left translation $L_{g_t}$ is an non-trivial isometry as well as a biholomorphism on the space $G/K$, that fixes the base point $\bar{e}$. So it suffices to study the index of the geodesic $L_{g_t{_*}}\dot{\gamma}$. For small values of t the set $\Gamma$ associated to this 
geodesic  will contain the long root $e_i -e_k$ by an application of a standard form relating the adjoint action of the Lie group and Lie algebra namely $Ad(exp(tX)Y = Y+t[X,Y]+O(t^2)$\cite{Hel} . So $\delta$ can be chosen in the form $e_p - e_q$ where $p < q$. We can now use the arguments above to prove the Theorem in the main case.

\noindent \textit{Part iv})

Using \cite{H} we note that the roots of $C_r$ are given by $$\Delta = \{ \pm e_i \pm e_j | 1 \leq i,j \leq r \} \cup \{ \pm 2e_i | 1 \leq i \leq r\}$$ where $e_i$ are the standard basis of $\mathbb{R}^r$. The positive roots are given by $$\Delta^+ = \{ e_i + e_j | 1 \leq i,j \leq r \} \cup \{ e_i - e_j | 1 \leq i < j \leq r \} \cup \{ 2e_i | 1 \leq i \leq r\}.$$

Assume we can choose superminimal $\delta \in \Gamma$ such that the root is long, in that case we can again use Proposition \ref{Weyl1} and from a simple calculation conclude that $\ell = r$. 

The only options left are $\delta = e_i+e_j$ or $e_i -e_j$ such that $i<j$. One can 
easily check that conditions 1 and 2 do not hold true for $\mathcal{T}$ and 
$\mathcal{S}$. We will use appropriate subsets $\mathcal{T}^*$ and $\mathcal{S}^*$ of $\mathcal{T}$ and $\mathcal{S}$ such that condition 1 and 2 are satisfied. It is easy to see that we can modify the proof of Theorem \ref{mainthm} by using these subsets.

Lets deal with the case where $\delta = e_i -e_j$ is superminimal for some $i<j$. Here we let $U_{\delta} = \{ e_k - e_j| k <i \} \cup \{ e_i -e_{k'}| j < k' \} \cup \{ 2e_i \} \cup \{ \delta \}$ and $V_{\delta}= \{ e_l - e_j| i<l<j \} \cup \{e_i -e_{l'}| i<l'<j \}$. 

Now $$T^*_{\delta} = T_{\delta} \cap (U_{\delta} \cup V_{\delta})$$

$$S^*_{\delta} = S_{\delta} \cap V_{\delta}$$. We can now verify that condition 1 and 2 are satisfied $T^*_{\delta}$ and $S^*_{\delta}$ 

Now we can assume that roots of the form $e_i -e_j \notin \Gamma$, for if such a root were in $\Gamma$ we could find a superminimal of that form, and proceed as in the earlier case. 

Now we assume that we can choose superminimal elements of the form $e_i + e_j$, here we let $U_{\delta} = \{ e_k + e_j| k <i \} \cup \{ 2e_i \}$ and $V_{\delta}= \{ e_j + e_{l}| i<l , l \neq j  \} \cup \{e_i -e_{l}| i<l, l \neq j\}$. We define $T^*_{\delta}$ and $S^*_{\delta}$ as we did previously and verify that conditions 1 and 2 are satisfied.

\section{Acknowledgment}
This work is based on the my's Phd Thesis at Michigan State University. I wish to thank Michigan State University for the support provided during the pursuit of my doctarate. I would like to thank my Advisor Prof Jon Wolfson for suggesting I investigate this topic and for the many helpfull discssions. 



\begin{thebibliography}{}
%
%
\bibitem{B} Barth, W. , Transplanting cohomology classes in complex projective space, Amer.
J. Math.92 (1970), 951-967
\bibitem{B-L} Barth, W. - Larsen, M., On the homotopy groups of complex projective manifolds,
Math. Scand.30 (1972), 88-94.
\bibitem{La} Larsen, M., On the topology of complex projective manifolds, Invent. Math. 19
(1973), 251-260.
\bibitem{S-W}Schoen, R., and Wolfson, J., Theorems of Barth-Lefschetz type and Morse Theory on the space of paths, Math Z. 229 (1998), 77-87
\bibitem{K-W}Kim, M., and Wolfson, J., Theorems of Barth-Lefschetz type on K\"{a}hler manifolds of non-negative bisectional curvature 
\bibitem{F} Frankel, T., Manifolds with positive curvature, Paciﬁc J. Math. 11 (1961), 165-174
\bibitem{S1} Sommese, A., Theorems of Barth-Lefschetz type for complex subspaces of homogeneous complex manifolds, Proc. Natl. Acad. Sci. USA 74 (1977) 1332-1333. 
\bibitem{S2} Sommese, A., Complex subspaces of homogeneous complex manifolds II- Homotopy Results, Nagoya Math. J. 86 (1982), 101-129 
\bibitem{G} Goldstein, N., Ampleness and connectedness in complex G/P, Trans. AMS, 274 (1982),361-373 
\bibitem{F-L} Fulton, W.- Lazarsfeld, R., Connectivity and its applications in algebraic geometry,Algebraic Geometry, Springer Lecture Notes in Math. 862 (1981), 26-92.
\bibitem{Le}Lefschetz, S.,L’Analysis Situs et la Geom´etrie Alg´ebrique´, Paris 1924.
\bibitem{H}Humphries, J. Introduction to Lie Algebras and Representation Theory, Springer, Heidelberg, 1972
\bibitem{K}Shoshichi Kobayashi and Katsumi Nomizu. Foundations of differential 
geometry. Vol. II. Interscience Publishers John Wiley and Sons, Inc.,
New York-London-Sydney, 1969. 
\bibitem{N-W}Ni, L., Wolfson, J.: The Lefschetz theorem for CR submanifolds and the nonexistence of real analytic Levi ﬂat submanifolds. Commun. Anal. Geom. 11, 553-564 (2003)
\bibitem{M}Milnor, J.,Morse Theory, Ann. of Math. Studies 51, Princeton Univ. Press, Princeton,
NJ, 1963. 
\bibitem{M1} Mok, N., Uniqueness Theorems of K\"{a}hler metrics of semipositive bisectional curvature on
compact hermitian symmetric spaces, Math. Ann. 276 (1987) 177-204 
\bibitem{O} A.L. Onishchik: Topology of Transitive Transformation Groups, Johann Ambrosius
Barth, Leipzig-Heidelberg-Berlin, 1994. 
\bibitem{S-V} Sommese, A., Van de Ven, A., Homotopy groups of pullbacks of varieties, Nagoya Math. J. 102 (1986), 79-90.
\bibitem{W}B. Wilking, Torus actions on manifolds of positive sectional curvature, Acta Mathematica, 191 (2003), 259-297
\bibitem{Wang} Wang, H.C.: Closed manifolds with a homogeneous complex structure. Am. J. Math. 76,1-32 (1954)
\bibitem{Chen} Xiaoyang Chen, Theorems of Barth-Lefschetz type in Sasakian geometry,arXiv:math.DG/1110.0565v1 
\bibitem{FM}Fang, F., Mendon¸ca, S., Complex immersions in K\"{a}hler manifolds of positive holomorphic k-Ricci curvature, Trans. Amer. Math. Soc. 357 (2005), no. 9,
3725–3738.
\bibitem{Fang} F. Fang, Positive quaternionic K\"{a}hler manifold and symmetry rank, Crelle’s Journal, 576(2004),149-165
\bibitem{N} K. Nomizu, Invariant affine connections on homogeneous spaces Amer.J.Math.,76 : 1 (1954) pp. 33–65
\bibitem{Hel} S. Helgason,Differential Geometry, Lie Groups, and Symmetric Spaces, Academic Press, New-York-San Francisco-London, 1978.
\bibitem{FMR}F.Fang and S.Mendonc¸a and X.Rong, A connectedness principle
in the geometry of positive curvature, Comm. Anal. Geom. 13
(2005), 671-695.

\end{thebibliography}
\end{document}